\documentclass[11pt]{article}
\usepackage{amsfonts}
\usepackage{amsmath}
\usepackage{amssymb}
\providecommand{\U}[1]{\protect\rule{.1in}{.1in}}
\providecommand{\U}[1]{\protect\rule{.1in}{.1in}}
\newcommand{\ov}{\overline}
\newcommand{\ro}{r_0}
\newcommand{\Cl}{\int_0^t}
\newcommand{\R}{{\mathbb R}}
\newcommand{\N}{{\mathbb N}}
\newcommand{\Rp}{{\mathbb R^+}}
\newcommand{\Rd}{{\mathbb R^d}}
\newcommand{\D}{\mathbb{D}(\mathbb{R}^{+},\mathbb{R}^{d})}
\newcommand{\DD}{\mathbb{D}(\mathbb{R}^{+},\mathbb{R}^{2d})}
\newcommand{\DDD}{\mathbb{D}(\mathbb{R}^{+},\mathbb{R}^{3d})}

\newcommand{\Dii}{{\mathbb D}\,(\Rp,\R^{2d})}

\newcommand{\Diii}{{\mathbb D}\,(\Rp,\R^{3d})}
\newcommand{\Diiii}{{\mathbb D}\,(\Rp,\R^{4d})}

\newcommand{\filtwh}{(\wh{\cal F}_t)}

\newcommand{\spawh}{(\wh \Omega,\,\wh{\cal F},\,(\wh{\cal F}_t),\,\wh P)}

\newcommand{\p}{{ P}}
\newcommand{\wh}{\widehat}

\newcommand{\arrowd}{\mathop{\longrightarrow}_{\cal D}}

\newcommand{\arrowp}{\mathop{\longrightarrow}_{\cal P}}

\newcommand{\cdw}{{\cal C}^2}

\newcommand{\lra}{\longrightarrow}
\newtheorem{theorem}{Theorem}[section]

\newtheorem{corollary}{Corollary}[section]

\newtheorem{definition}{Definition}[section]
\newtheorem{lemma}{Lemma}[section]

\newtheorem{proposition}{Proposition}[section]
\newtheorem{remark}{Remark}[section]

\newenvironment{proof}[1][Proof]{\noindent\textbf{#1.} }{\ \rule{0.5em}{0.5em}}

\renewcommand{\thefootnote}{\fnsymbol{footnote}}

\numberwithin{equation}{section}

\begin{document}

\title{Multivalued Monotone Stochastic Differential Equations\\ with Jumps}
\author{Lucian Maticiuc$^{a}$, Aurel R\u{a}\c{s}canu$^{a}$ and  Leszek S\l %
omi\'{n}ski$^{b,}$\thanks{
{\scriptsize Corresponding author.}} $\;$ \\
{\small $^{a}$ Faculty of Mathematics, \textquotedblleft Alexandru Ioan
Cuza\textquotedblright\ University, Carol 1 Blvd., no. 11, Ia\c{s}i, Romania,%
}\smallskip\\
{\small $^{b}$ Faculty of Mathematics and Computer Science, Nicolaus
Copernicus University},\\
{\small ul. Chopina 12/18, 87-100 Toru\'n, Poland} }
\date{}
\maketitle

\begin{abstract}
We study  multivalued stochastic differential equations (MSDEs)
with  maximal monotone operators driven by semimartingales with
jumps. We discuss in detail some  methods of approximation of
solutions of MSDEs based on discretization of processes and Yosida
approximation of the monotone operator. We also study the general
problem of stability  of solutions of MSDEs with respect to the
convergence of driving semimartingales.
\end{abstract}


\noindent\textbf{AMS Classification subjects}: Primary: 60H20; Secondary: 34A60.$%
\smallskip$

\noindent\textbf{Keywords or phrases}: Multivalued Stochastic
Differential Equations with Jumps; Maximal Monotone Operators;
Yosida Approximations.

\renewcommand{\thefootnote}{\fnsymbol{footnote}}
\footnotetext{\textit{\scriptsize E-mail addresses:} {\scriptsize %
lucian.maticiuc@ymail.com (Lucian Maticiuc), aurel.rascanu@uaic.ro (Aurel R%
\u{a}\c{s}canu), leszeks@mat.umk.pl (Leszek S\l omi\'{n}ski)}}

\section{Introduction}
Let $A:\mathbb{R}^{d}\rightarrow2^{\mathbb{R}^{d}}$ be a maximal monotone
multivalued operator on $\mathbb{R}^{d}$ with the domain $\mathrm{D}%
(A)=\{z\in\mathbb{R}^{d}:A(z)\neq\emptyset\}$ and its graph
\begin{equation*}
\mathrm{Gr}(A)=\{(z,y)\in\mathbb{R}^{2d}:z\in\mathbb{R}^{d},y\in A(z)\}.
\end{equation*}
Let $\Pi:\mathbb{R}^{d}\rightarrow\overline{\mathrm{D}(A)}$ be a \textit{%
generalized projection} on $\overline{\mathrm{D}(A)}$ (in the sense that $%
\Pi\left( x\right) =x$ for all $x\in\overline{\mathrm{D}(A)}$ and
$\Pi$ is a non--expansive map). In the paper we consider the
following $d$--dimensional MSDE driven by the operator $A$ and
associated with the projection $\Pi$:
\begin{equation}
X_{t}+K_{t}=H_{t}+\int_{0}^{t}\langle f(X_{s-}),dZ_{s}\rangle,\quad t\in%
\mathbb{R}^{+},   \label{eq1.1}
\end{equation}
where $Z$ is a $d$-dimensional semimartingale with $Z_{0}=0$, $H$ is a c%
\`{a}dl\`{a}g adapted process with $H_{0}\in\overline{\mathrm{D}(A)}=\mathrm{%
D}(A)\cup\mathrm{Bd}\left( \mathrm{D}(A)\right) $ and $f:\mathbb{R}%
^{d}\rightarrow\mathbb{R}^{d}\otimes\mathbb{R}^{d}$ is a
continuous function.
By a solution of (\ref{eq1.1}) we understand a pair $(X,K)$ of c\`{a}dl\`{a}%
g adapted processes such that $X_{t}\in\overline{\mathrm{D}(A)}$ for $%
t\in\mathbb{R}^{+}$, $K$ is a locally bounded variation process
such that $K_{0}=0$ and for any $(\alpha,\beta)\in\mathrm{Gr}(A)$,%
\begin{equation*}
\int_{s}^{t}\langle X_{u}-\alpha,dK_{u}^{c}-\beta\,du\rangle\geq0,\quad0\leq
s<t,\quad s,t\in\mathbb{R}^{+},
\end{equation*}
where $K_{t}^{c}:=K_{t}-\sum_{s\leq t}\Delta K_{s}$ and if
$|\Delta K_{t}|>0$ then
\begin{equation}
\label{eq1.2}
X_{t}=\Pi(X_{t-}+\Delta H_{t}+\langle f(X_{t-}),\Delta
Z_{t}\rangle ),\quad t\in \mathbb{R}^{+}
\end{equation}
(for the precise definition see Section 2).

Particular cases of the above type of MSDEs were considered earlier in many
papers. For instance, the existence and uniqueness of solutions of (\ref%
{eq1.1}) in the case of It\^{o} diffusions was proved independently in E. C%
\'{e}pa \cite{ce/98} and A. R\u{a}\c{s}canu \cite{ra/96} (in the
infinite dimensional framework). MSDEs with subdifferential
operator (i.e. with   maximal monotone operator of the form
$A=\partial\varphi$, where $\varphi$ is a proper convex and lower
semicontinuous function; see Remark \ref{rem2.1}(a) in the next
section) were studied  by A. R\u{a}\c{s}canu in \cite{ra/81}. More
recently, R Buckdahn et al. \cite{bu-ma-pa-ra/13}  extended the
results of  \cite{ra/81} to  non-convex setup by proving the
existence and uniqueness results both for the Skorokhod problem
and for the associated MSDE driven by the Fr\'{e}chet
subdifferential $\partial^{-}\varphi$ of a semiconvex function
$\varphi$.

In the case of It\^{o} diffusions, conditions ensuring existence,
uniqueness and convergence of approximation schemes  were given in
I. Asiminoaei, A. R\u{a}\c{s}canu \cite{as-ra/97}, V. Barbu, A.
R\u{a}\c{s}canu \cite{ba-ra/97}, A. Bensoussan, A. R\u{a}\c{s}canu
\cite{be-ra/97} and R. Pettersson \cite{pe/00}. SDEs with
subdifferential operator driven by general continuous
semimartingale were considered in A. Storm \cite{st/95}. The case
of diffusions with Poissonian jumps was considered by  C. Marois
\cite{ma} and quite recently by J. Wu \cite{wu} and A. Zalinescu
\cite{za}. They have imposed, however, a very restrictive
condition on the Poissonian measure coefficient, which forces that
$K$ is a process with continuous trajectories. As a result, in
proofs they can apply the  methods developed earlier for MSDEs
with continuous trajectories.

It is well known that for every nonempty closed convex set $D\subset\mathbb{%
R}^{d}$ its indicator function $\varphi={I}_{D}$ is a convex and
proper lower semicontinuous function (see Remark \ref{rem2.1}(b)).
This implies that equations (\ref{eq1.1}) are strongly related to
stochastic differential equations (SDEs) with reflecting boundary
condition in convex domains. Such type of equations were
introduced by A.V. Skorokhod \cite{sk,ss} in
one-dimensional case and $D=\mathbb{R}^{+}$. The case of
reflecting It\^{o} diffusions in convex domains $D$ was studied in
detail by T. Tanaka \cite{ta/79} and for
general, not necessary convex domains, by P-L. Lions, A.S. Sznitman \cite%
{li-sn/83}, A. Rozkosz \cite{Ro} and Y. Saisho \cite{sa/87}. W. \L aukajtys
\cite{la/04}  and L.
S\l omi\'{n}ski \cite{sl/93,sl/01}  considered SDEs with reflecting boundary conditions
in convex domain driven by a general semimartingale. Approximations of
solutions of SDEs with reflecting boundary condition were studied in D. L%
\'{e}pingle \cite{le/95}, J.L. Menaldi \cite{me/83}, R. Pettersson \cite%
{pe/95}, M. Bossy, E. Gobet, D. Talay \cite{bgt}, M. Bossy, M. Ciss\'e, D. Talay \cite{bct}, W. \L aukajtys, L. S\l omi%
\'{n}ski \cite{la-sl/03,la-sl/13}  and L. S\l omi\'{n}ski \cite{sl/94,sl/01}.  It is worth noting that in all
the papers devoted to reflecting
SDEs in convex domains the projection used is the classical one, i.e. if $%
|\Delta K_{t}|>0$ then (\ref{eq1.2}) is satisfied with $\Pi$
replaced by the classical classical projection on
$\overline{\mathrm{D}(A)}$, i.e.
$x=\Pi_{\overline{\mathrm{D}(A)}}(z)$ iff
$|z-x|=\inf\{|z-x^{\prime}|:x^{\prime}\in\overline{\mathrm{D}(A)}\}$.

In the present paper we study the existence, uniqueness,
approximations and stability of solutions of (\ref{eq1.1}) driven
by semimartingales with jumps. We assume that $A$ is a general
maximal monotone operator such that $\mathrm{Int}\left(
\mathrm{D}(A)\right) \neq\emptyset$ and  that the projection $\Pi$
is non--expansive. Since we consider generalized projections, our
results are new even in the case of SDEs with reflecting boundary
condition in convex domains.

The paper is organized as follows.  In  Section 2 we consider the
deterministic  Skorokhod problem with maximal monotone operator
and non-expansive projection. This problem was discussed  in
detail in the recent paper by  L. Maticiuc et al.
\cite{ma-ra-sl/13} (see Remark \ref{rem2.6}).  We refine slightly
compactness results from \cite{ma-ra-sl/13}.  With the use of the
so-called $\eta$-oscillations of real functions, we give new
estimates for solutions of the Skorokhod problem and then  we
apply them to prove new compactness criterion for these solutions
in $S$-topology introduced by Jakubowski \cite{ja} ($S$-topology
is weaker than the Skorokhod topology $J_1$).

Section 3 is devoted to the study  of strong solutions of
(\ref{eq1.1}). We prove the existence and uniqueness of a strong
solution  to (\ref{eq1.1}) provided that $f$ satisfies the linear
growth condition and is locally Lipschitz continuous. We propose
two practical schemes  of approximations  of (\ref{eq1.1}). The
first one is based on discrete approximations of processes $H$ and
$Z$ and is constructed with the analogy to the Euler scheme. We
prove its convergence in probability in the Skorokhod topology
$J_{1}$. The second scheme has the form
\begin{equation}
X_{t}^{n}+\int_{0}^{t}A_{n}(X_{s}^{n})ds=H_{t}+\int_{0}^{t}\langle
f(X_{s-}^{n}),dZ_{s}\rangle ,\quad t\in \mathbb{R}^{+},%
 \label{eq1.3}
\end{equation}%
where $A_n$, $n\in \mathbb{N}$, is the  Yosida approximation of
the operator $A$. We prove that for any stopping time $\tau $ such
that $\mathbb{P}(\tau <+\infty )=1$ and $\mathbb{P}(\Delta H_{\tau
}=\Delta Z_{\tau }=0)=1$, $X_{\tau }^{n}\xrightarrow[%
\mathcal{P}]{\;\;\;\;\;}X_{\tau }$ in probability, where $X$ is a
solution of (\ref{eq1.1}) associated with the maximal monotone
operator $A$ and the classical projection
$\Pi_{\overline{\mathrm{D}(A)}}$ ($X^{n}$ need not converge in
probability in the Skorokhod topology $J_{1}$). We also show that
a slightly modified Yosida type approximation converges to
solutions of (\ref{eq1.1}) with general non--expanding projection
$\Pi $.

In Section 4 we study the general problem of stability of
solutions (\ref{eq1.1}) with respect to the convergence of driving
semimartingales. Using new estimates from Section 2, we prove
stability results under the assumption that the sequence of driving
semimartingales satisfies the so-called condition (UT) introduced
by Stricker \cite{st} (see also \cite{ja-me-pa/89}). As a consequence,  we
show the existence of a weak solution of (\ref{eq1.1}) provided
that $f$ is continuous and satisfies the linear growth condition.

In the paper we consider the space $\mathbb{D}\left(\mathbb{R}%
^{+},\mathbb{R}^{d}\right) $  of all mappings $y:\mathbb{R}%
^{+}\rightarrow\mathbb{R}^{d}$ which are c\`{a}dl\`{a}g (right
continuous and admit left-hand limits) equipped with two different
topologies: the Skorokhod topology $J_{1}$ (for the definition and
many useful results on  $J_1$ topology see, e.g., S. Ethier and T.
Kurtz \cite{et-ku/86}  and J. Jacod and A. Shiryaev
\cite{ja-sh/87})  and the $S$-topology (see Jakubowski \cite{ja}).

For $x \in \D$, $\delta >0$, $T\in \Rp$ we denote by
$\omega'_x(\delta , q)$ the classical modulus  of continuity of
$x$ on $[0,T]$, i.e. $\omega'_x(\delta , T)=\inf \{ \max_{i \leq
r} \omega _x([t_{i-1},t_i)): 0=t_0 < \ldots < t_r =T$, $\inf_{i <
r}(t_i - t_{i-1}) \geq \delta\}$, where $\omega _x(I)=\sup_{s,t
\in I}|x_s-x_t|$. We set  $\left\Vert x\right\Vert _{\left[
s,t\right] }={\sup\limits_{r\in \left[ s,t\right] }|x_{r}|}$ and
$\left\Vert x\right\Vert _{t}=\left\Vert x\right\Vert _{\left[
0,t\right] }.$ Let $k:\left[ 0,T\right] \rightarrow\mathbb{R}^{d}$
and let $\mathcal{D}$ be the set of  partitions of the interval
$\left[ 0,T\right] $. For
$\Delta=\{0=t_{0}<t_{1}<\cdots<t_{n}=T\}$ we set $
V_{\Delta}(k)=\sum\limits_{i=0}^{n-1}|k(t_{i+1})-k(t_{i})|$ and
$\left\updownarrow k\right\updownarrow _{T}=\sup\limits_{\Delta\in\mathcal{D}%
}V_{\Delta}(k)$. Write $\mathrm{BV}(\left[ 0,T\right]
;\mathbb{R}^{d})=\{k:\left[ 0,T\right]
\rightarrow\mathbb{R}^{d}:\left\updownarrow k\right\updownarrow
_{T}<\infty\}$. We will say that $k\in\mathrm{BV}_{loc}(\mathbb{R}%
^{+};\mathbb{R}^{d})$ if, for every $T>0$, $k\in\mathrm{BV}(\left[
0,T\right] ;\mathbb{R}^{d})$. If $k$ is a function with locally
bounded variation then $ k_{t}^{c}=k_{t}-\sum_{s\leq t}\Delta
k_{s}$, $k_{t}^{d}=k_{t}-k_{t}^{c}$, $t\in\mathbb{R}^{+}$, and
$\left\updownarrow k\right\updownarrow _{\left[s,t\right]}$ stands
for its variation on $[s,t]$,  i.e.  $\updownarrow {k}
\updownarrow_{(t,T]}=\updownarrow {k}
\updownarrow_{T}-\updownarrow {k} \updownarrow_{t}$ with the
convention that $\updownarrow {k} \updownarrow_{0}=0$.

Let $Y=\{Y_{t}\}_{t\geq0}$ be an $(\mathcal{F}_{t})$-adapted process
and $\tau$ be an $%
(\mathcal{F}_{t})$-stopping time. We write $Y^{\tau}$ and $Y^{\tau-}$ to
denote the stopped processes $Y_{\cdot\wedge{\tau}}$ and $Y_{\cdot\wedge{%
\tau-}}$, respectively. Given a semimartingale $Y$ we denote by
$[Y]$ its quadratic variation process and by $\langle Y\rangle$
the predictable compensator of $[Y]$.  By ${\displaystyle
\arrowd}$ and ${\displaystyle\arrowp}$ we denote the convergence
in law and in probability, respectively.

\section{Preliminaries. The Skorokhod problem}

A set--valued operator $A$ on ${{\mathbb{R}^{d}}}$ is said to be monotone if%
\begin{equation*}
\langle y-y^{\prime},z-z^{\prime}\rangle\geq0,\quad\forall~(z,y),(z^{\prime
},y^{\prime})\in\mathrm{Gr}(A)
\end{equation*}
and $A$ is said to be maximal monotone if the condition $\langle
y-v,z-u\rangle\geq0$, $\forall~(u,v)\in\mathrm{Gr}(A)$ implies that $(z,y)\in%
\mathrm{Gr}(A)$.

\begin{remark}[see \protect\cite{br/73}]
{\rm \label{rem2.1}

(a) Let $\varphi:\mathbb{R}^{d}\rightarrow \mathbb{R}%
\cup\{+\infty\}$ be a proper convex and lower semicontinuous function. The
subdifferential operator of $\varphi$ is defined by%
\begin{equation*}
\partial\varphi(z):=\{y\in{{\mathbb{R}^{d}}}:\langle
y,z^{\prime}-z\rangle+\varphi(z)\leq\varphi(z^{\prime}),\,\forall~z^{\prime}%
\in{{\mathbb{R}^{d}}}\},\quad z\in{{\mathbb{R}^{d}}}
\end{equation*}
and it is a maximal monotone operator on $\mathbb{R}^{d}$.

(b) Let $D$ be a closed convex nonempty subset of $\mathbb{R}^{d}$
and let $I_{D}:\mathbb{R}^{d}\rightarrow\mathbb{R}\cup\{+\infty\}$ be the
convexity indicator function (i.e. $I_{D}\left( z\right) =0$ if $z\in D$ and
$+\infty$, otherwise). Then $I_{D}$ is convex lower semicontinuous and
proper. Moreover, $\partial I_{D}(z)=\emptyset$ if $z\notin D$ and%
\begin{equation*}
\partial{I}_{D}(z)=\{\langle y,x-z\rangle\leq0,\,x\in D\},\quad z\in D,
\end{equation*}
which implies that%
\begin{equation*}
\partial I_{D}(z)=\left\{
\begin{array}{ll}
\{0\}, & \text{if }z\in\mathrm{Int}(D),\medskip \\
N_{D}(z), & \text{if }z\in\mathrm{Bd}\left( D\right) .%
\end{array}
\right.
\end{equation*}
Here $N_{D}(z)$ denotes the closed external cone normal to $D$ at $z\in%
\mathrm{Bd}\left( D\right) $.
}
\end{remark}

In the paper we will restrict our attention to operators $A$ and
projections $\Pi$ satisfying the following hypotheses
\begin{description}
\item[(H1)] $A$ is  a maximal monotone operators $A$ such
that
\begin{equation}
\mathrm{Int}\left( \mathrm{D}(A)\right) \neq\emptyset,
\label{eq2.1}
\end{equation}
\item[(H2)]$\Pi:{{\mathbb{R}^{d}}}\rightarrow\overline{\mathrm{D}(A)}$
is a generalized projection such that
\begin{equation}
\left\{
\begin{array}{l}
\Pi(z)=z,~\forall z\in\overline{\mathrm{D}(A)},\medskip \\
|\Pi(z)-\Pi(z^{\prime})|\leq|z-z^{\prime}|,\quad\forall z,z^{\prime}\in{{%
\mathbb{R}^{d}}}.%
\end{array}
\right.   \label{eq2.2}
\end{equation}
\end{description}

It is well known that $\overline{\mathrm{D}(A)}$ is convex (see, e.g., \cite%
{br/73}). Let $\Pi_{\overline{\mathrm{D}(A)}}$ denote the
classical
projection on $\overline{\mathrm{D}(A)}$ with the convention that
$\Pi_{\overline {%
\mathrm{D}(A)}}(z)=z$, $\forall z\in\overline{\mathrm{D}(A)}$.
One can check
that%
\begin{equation*}
x=\Pi_{\overline{\mathrm{D}(A)}}(z)\Leftrightarrow\langle z-x,x^{\prime
}-x)\rangle\leq0,~\forall x^{\prime}\in\overline{\mathrm{D}(A)}
\end{equation*}
and%
\begin{equation}
|\Pi_{\overline{\mathrm{D}(A)}}(z)-\Pi_{\overline{\mathrm{D}(A)}}(z^{\prime
})|^{2}\leq\langle\Pi_{\overline{D(A)}}(z)-\Pi_{\overline{\mathrm{D}(A)}%
}(z^{\prime}),z-z^{\prime}\rangle,\quad\forall z,z^{\prime}\in{{\mathbb{R}%
^{d}}},   \label{eq2.3}
\end{equation}
which implies (\ref{eq2.2}).
There exist other important examples of  non-expanding projections on $%
\overline{\mathrm{D}(A)}$ associated with the \textit{elasticity
condition} (introduced in the one-dimensional case in
\cite{ch-ka/80} and \cite{sl-wo/10}): let $c\in\lbrack0,1]$ and
let $\Pi^{c,n}:{{\mathbb{R}^{d}}}\rightarrow{{\mathbb{R}^{d}}}$ be
of the form $ \Pi^{c,n}(z)=\Pi_{n}\circ...\circ\Pi_{1}(z)$,
$z\in{{\mathbb{R}^{d}}}$, where
$\Pi_{1}=\ldots=\Pi_{n}=\Pi^{c},\;n\in{\mathbb{N}}$ and
$\Pi^{c}:{{\mathbb{R}^{d}}}\rightarrow{{\mathbb{R}%
^{d}}}$ be given by %
\begin{equation*}
\Pi^{c}(z):=\Pi_{\overline{\mathrm{D}(A)}}(z)-c\big(z-\Pi_{\overline {%
\mathrm{D}(A)}}(z)\big),\quad z\in{{\mathbb{R}^{d}}}.
\end{equation*}
It can be shown (see \cite[Proposition 11]{ma-ra-sl/13}) that
there exists the limit
$\bar\Pi(z)=\lim_{n\rightarrow\infty}\Pi^{c,n}(z)$
and $\bar\Pi(z)$ is a generalized projection satisfying (\ref{eq2.2})).$%
\smallskip$

\begin{definition}[see \protect\cite{ma-ra-sl/13} Definition 14]\label{def2.5}
{\rm Let $y\in\mathbb{D}(\mathbb{R}^{+},\mathbb{R}^{d})$ be such that
$y_{0}\in\overline{\mathrm{D}(A)}$. We say that a~pair $(x,k)\in
\mathbb{D}(\mathbb{R}^{+},\mathbb{R}^{2d})$ is a solution of the
Skorokhod problem  associated with $y$, the maximal monotone
operator $A$ and the projection $\Pi$
($(x,k)=\mathcal{SP}(A,\Pi;y)$ for short) if
\begin{description}
\item[(i)] $x_{t}=y_{t}-k_{t}\in\overline{\mathrm{D}(A)}$,$~t\in \mathbb{R}%
^{+},$

\item[(ii)] $k$ is a function with locally bounded variation such that $%
k_{0}=0$ and for any $(\alpha,\beta)\in\mathrm{Gr}(A)$,
\begin{equation*}
\int_{s}^{t}\langle x_{u}-\alpha,dk_{u}^{c}-\beta\,du\rangle\geq0,\quad0\leq
s<t,\,\,s,t\in{\mathbb{R}^{+}},
\end{equation*}

\item[(iii)] if $|\Delta k_{t}|>0$ then
$x_{t}=\Pi(x_{t-}+\Delta y_{t})$, $t\in\mathbb{R}^{+}$.
\end{description}
}
\end{definition}

\begin{remark}[\protect\cite{ma-ra-sl/13} Lemma 20, 23, Theorem 24]
{\rm Assume (H1), (H2).

(a) \label{rem2.6} For every
$y\in{\mathbb{D}}({\mathbb{R}^{+}},{{\mathbb{R}}}^{d})$ such that
$y_{0}\in\overline{\mathrm{D}(A)}$ there exists a unique solution
of the Skorokhod problem associated with $y$, the maximal monotone
operator $A$ and the projection $\Pi$. Since
$x_{t-}\in\overline{\mathrm{D}(A)}$ and $\Pi$ is non--expansive,
$|\Delta k_{t}|\leq2|\Delta y_{t}|$, $t\in{\mathbb{R}^{+}}$.

(b) Let $y,y^{\prime}\in{\mathbb{D}}({\mathbb{R}^{+}},{{%
\mathbb{R}}}^{d})$ be  such that
$y_{0},y_{0}^{\prime}\in\overline{D(A)}$.
If $%
(x,k)=\mathcal{SP}(A,\Pi;y)$ and $(x^{\prime},k^{\prime})=\mathcal{SP}%
(A,\Pi;y^{\prime})$ then
for every $t\in{\mathbb{R}^{+}}$,%
\begin{equation*}
|x_{t}-x_{t}^{\prime}|^{2}\leq|y_{t}-y_{t}^{\prime}|^{2}-2\int_{0}^{t}%
\langle
y_{t}-y_{t}^{\prime}-y_{s}+y_{s}^{\prime},dk_{s}-dk_{s}^{\prime}\rangle.
\end{equation*}

(c) Let $a\in \mathrm{Int}\left( \mathrm{D}\left( A\right) \right)
$ and $r_{0}>0$ be such that $\overline{B\left( a,r_{0}\right)
}\subset \mathrm{D}\left( A\right)$.
If%
\begin{equation*}
\sup \left\{ \left\vert \hat{u}\right\vert :\hat{u}\in Au,\;u\in
\overline{B\left( a,r_{0}\right) }\right\} \leq \mu <\infty,
\end{equation*}%
then for every $0\leq s<t$,%
\begin{equation*}
r_{0}\left\updownarrow k\right\updownarrow _{\left[ s,t\right] }\leq
\int_{s}^{t}\langle x_{r}-a,dk_{r}\rangle +\frac{1}{2}\sum_{s<r\leq
t}|\Delta k_{r}|^{2}+\mu \int_{s}^{t}|x_{r}-a|dr+(t-s)r_{0}\mu \,.
\end{equation*}

(d)  Let $%
(x^{n},k^{n})=\mathcal{SP}(A,\Pi;y^{n})$, $ n\in{\mathbb{N}}$.
 If $||y^n-y||_T\longrightarrow0$, $T>0$ (resp. $y^{n}\longrightarrow y$ in ${\mathbb{D}}({%
\mathbb{R}^{+}},{{\mathbb{R}}}^{d})$) then%
\begin{align*}
||x^n-x||_T\longrightarrow0\quad\mbox{\rm and}&\quad||k^n-k||_T\longrightarrow0,\quad T>0\\
(\mbox{\rm resp. }\,(x^{n},k^{n},y^{n})\longrightarrow&(x,k,y)\quad\text{in }{\mathbb{D}}({%
\mathbb{R}^{+}},{\mathbb{R}}^{3d})),
\end{align*}
where $(x,k)=\mathcal{SP}(A,\Pi;y)$.}
\end{remark}

We now  give some new estimates for solutions of the Skorokhod
problem. These estimates will play a key role in our proofs in
Section 4. We recall that for $y\in\D$, $\eta>0$ and $T\in\Rp$ the
number $N_{\eta}$ of $\eta$-oscillations is defined as follows:
$N_{\eta}(y,T)\geq k$ if one can find numbers $0\leq t_1\leq
t_2\leq\dots\leq t_{2k-1}\leq t_{2k}\leq T$ such that
 $|y_{t_{2i-1}}-y_{t_{2i}}|>\eta$, $i=1,2,\dots,k$.
\begin{proposition}
\label{prop2.11} Assume (H1), (H2).
 Let $(x,k)$ be a solution of the Skorokhod problem
associated with $y$, $y_0\in \ov{D(A)}$, $A$  and  $\Pi$. Then for
any $a \in \mathrm{Int}(D(A))$, $T\in\Rp$ and the constants
$r_0,\mu$  from  Remark \ref{rem2.6}(c) there exist constants
$C_1,C_2>0$ depending also on $N_{r_0/2}(y,T)$ such that
\begin{enumerate}
\item[\rm{(i)}] $\displaystyle{
\|x\|_T \leq C_1(1+\|y\|_T),}$
\item[\rm{(ii)}] $\displaystyle{\left\updownarrow k\right\updownarrow_{T}
\leq C_2 (1+\|y\|_T^2)}$.
\end{enumerate}
\end{proposition}

\begin{proof} (i)
The proof is similar to that of \cite[Theorem 4.8]{ce/98}.  It is
easily seen that for any $0\leq t\leq T$,
\begin{eqnarray*}
& &|x_t-a|^2  = |y_t-a|^2 + \langle k_t,k_t\rangle -2 \Cl \langle
y_t-a,dk_u
\rangle\\
 & &\qquad= |y_t-a|^2 + 2\Cl \langle k_{u}, dk_u
 \rangle - \sum_{u\leq t}|\Delta k_u|^2-2 \Cl
\langle y_t-a,dk_u \rangle \\
  & &\qquad =  |y_t-a|^2-2\Cl \langle x_u-a,dk_u\rangle
  -2\Cl \langle y_t-y_u,dk_u\rangle
  -  \sum_{u\leq t}|\Delta k_u|^2.
\end{eqnarray*}
Hence
\begin{eqnarray*}
&&|x_t-a|^2-|x_s-a|^2
  =  |y_t-a|^2 - |y_s-a|^2
-2\int_s^t\langle x_u-a,dk_u \rangle \\
& &\qquad\qquad\qquad + 2\int_s^t \langle y_u-y_s,dk_u \rangle
 - 2 \langle k_t,y_t-y_s \rangle -
\sum_{s < u \leq t}|\Delta k_u|^2.
\end{eqnarray*}
for $0 \leq s \leq t \leq T$. By Remark \ref{rem2.6}(c), for any
$0\leq s\leq t\leq T$ we have
\begin{eqnarray}
|x_t-a|^2-|x_s-a|^2 & \leq&
\nonumber |y_t-a|^2 - |y_s-a|^2- 2 r_0
\left\updownarrow k\right\updownarrow _{\left[ s,t\right] }
+ 2\int_s^t\langle
y_u-y_s,dk_u \rangle \\& &\quad +2\langle a-y_t,y_t-y_s
\rangle + 2 \langle x_t-a,y_t-y_s \rangle\nonumber \\
&&\quad+2\mu\int_s^t|x_u-a|du+2\gamma\mu(t-s)\nonumber\\
 &\leq &\label{eq2.9}5\bar y^2 + 4\bar y\bar x
-2r_0\left\updownarrow k\right\updownarrow _{\left[ s,t\right] }
+ 2\int_s^t\langle y_u-y_s,dk_u \rangle\\
&&\quad\nonumber +2\mu\int_s^t|x_u-a|du+2\gamma\mu(t-s),
\end{eqnarray}
where $\bar  y=\|y-a\|_T$, $\bar  x=\|x-a\|$. Set $s_0=0$ and
$s_j=\inf\{s>s_{j-1};|y_s-y_{s_{j-1}}|>r_0/2\}\wedge T$, $j\in\N$.
Let $m=\inf\{j;s_j=T\}$. Since $y\in\D$, it follows easily that
$m<\infty$. Moreover, on each interval $[s_{j-1},s_j]$ there is at
least one $r_0/2$-oscillation of $y$ and hence $m\leq
N_{r_0/2}(y,T)+1$. Observe that for every $j=1,\dots,m$,
\begin{align*}
\int_{s_{j-1}}^{s_j}\langle y_u-y_{s_{k j-1}},dk_u \rangle &=
\int_{(s_{j-1},s_{j})} \langle y_u-y_{s_{j-1}},dk_u \rangle +
\langle \Delta y_{s_j} , \Delta k_{s_j} \rangle\\
 &\leq \frac{r_0}{2}\left\updownarrow k\right\updownarrow _{\left[ s_{j-1},s_j\right] } +
\langle \Delta y_{s_j} , \Delta k_{s_j} \rangle,
\end{align*}
which implies that
\begin{align}
2(\int_{s_{j-1}}^{s_j} \langle
y_u-y_{s_{j-1}},dk_u \rangle - r_0\left\updownarrow
k\right\updownarrow _{\left[ s_{j-1},s_j\right] })
  &\leq
  -\ro\left\updownarrow k\right\updownarrow _{\left[ s_{j-1},s_j\right] } +
2\langle \Delta y_{s_j} , \Delta k_{s_j} \rangle\nonumber\\
&\leq
\label{eq2.10}
-r_0\left\updownarrow k\right\updownarrow
_{\left[ s_{j-1},s_j\right] } +16\bar y^2,
\end{align}
because $|\Delta k_{s_j}|\leq 2|\Delta y_{s_j}|\leq4\bar y$,
$j=1,\dots,m$. Putting $s=s_{j-1}$, $t=s_j$ in (\ref{eq2.9}) and
applying (\ref{eq2.10}) we obtain
\begin{align}
|x_{s_j}&-a|^2-|x_{s_{j-1}}-a|^2  \nonumber\\
& \leq 21\bar y^2+ 4\bar y\bar x
 -r_0\left\updownarrow k\right\updownarrow _{\left[ s_{j-1},s_j\right] }
 +2\mu\int_{s_{j-1}}^{s_j}|x_u-a|du+2r_0\mu(s_j-s_{j-1}).\label{eq2.11}
\end{align}
Set $m_0=\max \{ j, s_j \leq t \}$ and observe that  from
(\ref{eq2.11}) it  follows that
\begin{eqnarray*}
&&\!\!|x_t-a|^2  =
\sum_{j=1}^{m_0}(|x_{s_j}-a|^2-|x_{s_{j-1}}-a|^2 )
+|x_t-a|^2-|x_{s_{j_0}}-a|^2+|x_0-a|^2\\
 & &\,\qquad\quad\leq  m(21\bar y^2+ 4\bar y\bar x)+\bar y^2+
 +2\mu\int_0^t|x_u-a|du+2r_0\mu t.
\end{eqnarray*}
Hence $\bar x^2 \leq  38m^2\bar y^2+ {\bar x^2}/2
+(4\mu^2T^2+2r_0\mu T)$, which implies that
\begin{equation}
\label{eq2.12} \bar x^2 \leq 76m^2\bar y^2+(8\mu^2T^2+4r_0\mu T).
\end{equation}
This completes the proof of (i).
\\
(ii) By (\ref{eq2.11}),
\begin{eqnarray*}
\ro\left\updownarrow k\right \updownarrow _{\left[
s_{j-1},s_j\right] }  &\leq& 21\bar y^2+ 4\bar y\bar x
+2\mu\int_{s_{j-1}}^{s_j}|x_u-a|du+2\gamma\mu(s_j-s_{j-1})
+|x_{s_{j-1}}-a|^2\\
&\leq&
 37\bar y^2+ 2\bar x^2+2\mu^2(s_j-s_{j-1})^2+2r_0\mu(s_j-s_{j-1})
\end{eqnarray*}
for   $j=1,\dots ,m$. Hence
\[ \left\updownarrow k\right\updownarrow _{T } \leq
\sum_{j=1}^{m}\left\updownarrow k\right \updownarrow _{\left[
s_{j-1},s_j\right] } \leq m(37\bar y^2+ 2\bar
x^2)+2\mu^2q^2+2r_0\mu T,
\]
which when combined with (\ref{eq2.12}) proves (ii).
\end{proof}
\begin{corollary}\label{cor2.12}Assume (H1), (H2).
Let $(x,k)$ be a solution of the Skorokhod problem associated with
$y$, $y_0\in \ov{D(A)}$, $A$ and $\Pi$. Let $a \in \mathrm{Int}(D(A))$ and let $\ro,\mu$ denote the constants from Remark
\ref{rem2.6}(c). Set $t'=\inf\{t;|y_t-y_0|\geq \ro/2\}\wedge T$.
Then
\[
\|x-a\|_{t'-}^2\leq  76\|y-a\|^2_{t'-}+(8\mu^2T^2+4\ro\mu T).
\]
\end{corollary}
\begin{proof} It suffices to put  $m=1,s_0=0,s_1=t'$
in  (\ref{eq2.12}).
\end{proof}
\medskip

In  Remark \ref{rem2.6}(d) we recalled  a stability result in
topology $J_1$ for solutions of the Skorokhod problem.  On the
space $\D$ one can also consider  the so-called $S$-topology
introduced by Jakubowski \cite{ja}. By \cite{ja}, $\{y^n\}$ is
relatively $S$-compact if and only if
\begin{equation}\label{eq6.1}
\sup_{n}\|y^n\|_T<+\infty, \quad T\in\Rp
\end{equation}
and for every $\eta>0$,
\begin{equation}\label{eq6.2}
\sup_{n}N_{\eta}(y^n,T)<+\infty,\quad T\in\Rp.
\end{equation}
We also recall  that $y^n$ converges to $y$ in  $S$-topology if
and only if $\{y^n\}$ satisfies (\ref{eq6.1}), (\ref{eq6.2})  and
from every subsequence $\{n_k\}$  one can choose  a further
subsequence $\{n_{k_l}\}$ such that $y^{n_{k_l}}_t\rightarrow y_t$
for every $t$  from some dense subset $Q\subset \Rp$.

\begin{corollary}\label{cor2.14} Assume (H1), (H2).
Let  $%
(x^{n},k^{n})=\mathcal{SP}(A,\Pi;y^{n})$, $n\in{\mathbb{N}}$.
If $\{y^{n}\}$ is relatively compact in the $S$-topology on ${\mathbb{D}}({%
\mathbb{R}^{+}},{{\mathbb{R}}}^{d})$ then%
\begin{equation*}
\{(x^{n},k^{n},y^{n})\}\quad\text{is relatively compact in the $S$-topology on }{\mathbb{D}}({%
\mathbb{R}^{+}},{\mathbb{R}}^{3d}).
\end{equation*}
\end{corollary}
\begin{proof}
By Proposition \ref{prop2.11}(ii), $\{k^n\}$ is relatively
$S$-compact. Hence  $\{x^n=y^n-k^n\}$ is relatively $S$-compact,
which completes the proof.
\end{proof}

\begin{remark}
{\rm The convergence in  topology  $J_1$  in Remark
\ref{rem2.6}(d) can not be replaced by the convergence in
$S$-topology. Convergence in $S$-topology does not hold even in
the case of the  classical Skorokhod problem. To see this, let us
consider the following example. Let $d=1$, $D=\Rp$, $A=\partial
I_D$, $\Pi=\Pi_{\ov{D(A)}}$. Set $y^n_t=0$ for  $t<1$  and $t\geq
1+1/n$,  and $y^n_t=-1$ for $t\in[1,1+1/n)$, $n\in\N$. Then
$y^n\rightarrow 0$
in  $S$-topology.  On the other hand,   $%
(x^{n},k^{n})=\mathcal{SP}(A,\Pi;y^{n})$ has the form
\begin{equation*}
x^n_t=\left\{
\begin{array}{ll}
0 & \text{if }t<1+1/n,\smallskip \\
1, & \text{if }t\geq 1+1/n,%
\end{array}
\right.
\qquad
k^n_t=\left\{
\begin{array}{ll}
0 & \text{if }t<1,\smallskip \\
1, & \text{if }t\geq 1%
\end{array}
\right.
\end{equation*}
for $n\in\mathbb{N}$, so does not converge to
$(0,0)=\mathcal{SP}(A,\Pi;0)$. }
\end{remark}

\section{MSDEs with maximal monotone operators}

Let $\left( \Omega,\mathcal{F},\mathbb{P},(\mathcal{F}_{t})_{t\geq
0}\right) $ be a stochastic basis, i.e. $\left( \Omega,\mathcal{F},\mathbb{P}%
\right) $ is a complete probability space and
$(\mathcal{F}_{t})_{t\geq0}$ is a filtration (an increasing
collection of completed $\sigma$-algebras of $\mathcal{F}$). Let
$Y$ be an $\mathcal{F}_{t}$-adapted  c\`{a}dl\`{a}g stochastic
process such that $Y_{0}\in\overline{\mathrm{D}\left( A\right) }$.

\begin{definition}
\label{def4.1}{\rm We say that a pair $(X,K)$ of $(\mathcal{F}_{t}$)-adapted c\`{a}%
dl\`{a}g stochastic processes is a solution of the Skorokhod
problem  associated with $Y$, the maximal monotone operator $A$
and the projection $\Pi$ ($(X,K)=\mathcal{SP}(A,\Pi;Y)$ in
notation), if  $(X\left( \omega\right) ,\,K\left( \omega\right)
)=\mathcal{SP}(A,\Pi;Y\left( \omega\right) )$ for
$\mathbb{P}$-a.e. $\omega\in\Omega$.
}
\end{definition}

In the case where  $(X,K)=\mathcal{SP}(A,\Pi;Y)$ we will sometimes write $%
X=\mathcal{SP}^{(1)}(A,\Pi;Y)$ and $K=\mathcal{SP}^{(2)}(A,\Pi;Y)$. If $Y$
is continuous then $(X,K)$ is also continuous and does not depend on $\Pi$.
In this case we write $(X,K)=\mathcal{SP}(A;Y)$ and $X=\mathcal{SP}%
^{(1)}(A;Y)$, $K=\mathcal{SP}^{(2)}(A;Y)$.

\begin{proposition}\label{prop1}
Assume (H1), (H2).  If $
Y:\Omega\times\mathbb{R}^{+}\rightarrow\mathbb{R}^{d}$ is an $(\mathcal{F}_{t})$%
-adapted c\`{a}dl\`{a}g  process such that $Y_{0}\in\overline {%
\mathrm{D}\left( A\right)}$, then there exists a unique pair
$\left( X,K\right) $ of $(\mathcal{F}_{t})$-adapted c\`{a}dl\`{a}g
processes which is a solution of  the Skorokhod problem associated
with $Y$, the maximal monotone operator $A$ and the projection
$\Pi$.
\end{proposition}
\begin{proof}
By Remark \ref{rem2.6}(a)  for each $\omega\in\Omega$ there exists
a unique pair $\left( X\left( \omega\right) ,K\left( \omega\right)
\right)$ such that
$X\left( \omega\right) \in\mathbb{D}\left( \mathbb{R}^{+},\mathbb{R}%
^{d}\right)$, $K\left( \omega\right) \in\mathbb{D}\left( \mathbb{R}^{+},%
\mathbb{R}^{d}\right) \cap\mathrm{BV}\left( \left[ 0,T\right] ;\mathbb{D}%
\left( \mathbb{R}^{+},\mathbb{R}^{d}\right) \right)$ and $\left(
X\left( \omega\right) ,K\left( \omega\right) \right)
=\mathcal{SP}(A,\Pi;Y\left( \omega\right) )$. What is left is to
show that $\left( X,K\right) $ is adapted. Let $\{\pi
_{n}=\{0=t_{n0}<t_{n1}<\ldots<t_{nk}<\ldots\}\}$ be a sequence of
partitions of ${\mathbb{R}^{+}}$ such that $\lim
_{n\rightarrow\infty}\max(t_{nk}-t_{n,k-1}) =0$.
Let $Y_{t}^{(n)}=Y_{t_{nk}}$, $t\in\lbrack t_{nk},t_{n,k+1})$, $k\in {%
\mathbb{N}}$, $n\in\mathbb{N}$, denote the sequence of
discretizations of $Y$ and let
$X^{(n)}=\mathcal{SP}(A_{n},\Pi;Y^{(n)})$. By \cite[Lemma 22]{ma-ra-sl/13} the process $X^{n}$ is given by the
formula%
\begin{equation*}
X_{t}^{(n)}=\left\{
\begin{array}{ll}
\mathcal{SP}^{\left( 1\right) }(A,Y_{0})_{t}, & t\in\lbrack0,t_{n,1}), \\
\mathcal{SP}^{\left( 1\right) }\big(A,\Pi(X_{t_{{n,k}%
-}}^{(n)}+Y_{t_{n,k}}-Y_{t_{n,k-1}})\big)_{t-t_{n,k}}, & t\in\lbrack
t_{n,k},t_{n,k+1}),\,k\in{\mathbb{N}},%
\end{array}
\right.
\end{equation*}
and hence is $(\mathcal{F}_{t})$-adapted. Since by \cite[Chapter 3
Proposition 6.5]{et-ku/86}, $
Y^{(n)}\rightarrow Y$ in ${\mathbb{D}}({\mathbb{R}^{+}},{{\mathbb{R}}}%
^{d})$, $\mathbb{P}$-a.s, it follows from Remark \ref{rem2.6}(d)
that
\begin{equation*}
X^{(n)}\rightarrow X\quad\text{in }{\mathbb{D}}({\mathbb{R}^{+}},{{%
\mathbb{R}}}^{d}),~\mathbb{P}\text{--a.s.}
\end{equation*}
Therefore the limit process $X$ is $(\mathcal{F}_{t})$-adapted.
Since $K=Y-X$, $K$ is also $(\mathcal{F}_{t})$-adapted.\hfill
\end{proof}
\medskip

Let $Y,\hat{Y}$ be two $(\mathcal{F}_{t})$-adapted processes with
trajectories in ${\mathbb{D}}\left( {\mathbb{R}^{+}},\mathbb{R}^{d}\right) $
admitting decompositions%
\begin{equation}
Y_{t}=Y_{0}+H_{t}+M_{t}+V_{t}\,,\quad\hat{Y}_{t}=\hat{Y}_{0}+H_{t}+\hat{M}%
_{t}+\hat{V}_{t},\quad t\in{\mathbb{R}^{+}},   \label{eq4.1}
\end{equation}
with $Y_{0}+H_{0}$, $\hat{Y}_{0}+H_{0}\in\overline{\mathrm{D}\left( A\right)
}$, where $H$ is an $(\mathcal{F}_{t})$-adapted process with trajectories in $%
{\mathbb{D}}\left( {\mathbb{R}^{+}},{{\mathbb{R}^{d}}}\right) $, $M$, $\hat{M%
}$ are $(\mathcal{F}_{t})$-adapted local martingales and $V,\hat{V}$ are $%
(\mathcal{F}_{t})$-adapted processes with bounded variation such that $M_{0}=%
\hat{M}_{0}=V_{0}=\hat{V}_{0}=0$.

\begin{lemma}
\label{lem4.3}Assume (H1), (H2).  Let $(X,K)=\mathcal{SP}(A,\Pi;Y)$
and $(\hat{X},\hat {K})=%
\mathcal{SP}(A,\Pi;\hat{Y})$. Then for any $p\in{\mathbb{N}}$
there exists a constant $C_{p}>0$ such that for every stopping
time $\tau$,
\begin{enumerate}\item[\rm(i)]
$\mathbb{E}{\|X-\hat{X}\|_{\tau}^{2p}\leq C_{p}\,\mathbb{E%
}}\Big({|Y_{0}-\hat{Y}_{0}|^{2p}+[M-\hat{M}]_{\tau}^{p}+\updownarrow{\hspace{-0.1cm} V-\hat {V}%
\hspace{-0.1cm}}\updownarrow_{\tau}^{2p}}\Big)$,
\item[\rm(ii)] $\mathbb{E}{\|X-\hat{X}\|^{2p}_{\tau-}\leq
C_{p}\,\mathbb{E}\Big(|Y_{0}-\hat{Y}_{0}|^{2p}+[M-\hat{M}]_{\tau-}^{p}+%
\langle M-\hat{M}\rangle_{\tau-}^{p}+\updownarrow{\hspace{-0.1cm}V-\hat{V}\hspace{-0.1cm}}\updownarrow_{\tau-}^{2p}}\Big)$.%
\end{enumerate}
\end{lemma}
\begin{proof}
By Remark \ref{rem2.6}(b), for every $t\in{\mathbb{R}^{+}}$,
\begin{equation}
|X_{t}-\hat{X}_{t}|^{2}\leq|Y_{t}-\hat{Y}_{t}|^{2}-2\int_{0}^{t}\langle
Y_{t}-\hat{Y}_{t}-Y_{s}+\hat{Y}_{s},dK_{s}-d\hat{K}_{s}\rangle.\label{eq
Lemma 4.3}
\end{equation}
By the equality
$\int_{0}^{t}\langle Y_{s}-\hat{Y}_{s},d(K_{s}-\hat{K}_{s})\rangle=%
\int_{0}^{t}\langle Y_{s-}-\hat{Y}_{s-},d(K_{s}-\hat{K}_{s})\rangle+[Y-\hat{Y%
},K-\hat{K}]_{t}$ and the integration by parts formula,
\begin{align*}
2\int_{0}^{t}\langle &Y_{t}-\hat{Y}_{t}-Y_{s}+\hat{Y}%
_{s},dK_{s}-d\hat{K}_{s}\rangle\medskip\\&=2\int_{0}^{t}\langle Y_{s-}-\hat{Y}%
_{s-},dY_{s}-d\hat{Y}_{s}\rangle-2\int_{0}^{t}\langle X_{s-}-\hat{X}%
_{s-},dY_{s}-d\hat{Y}_{s}\rangle\medskip \\
&=|Y_{t}-\hat{Y}_{t}|^{2}-|Y_{0}-\hat{Y}_{0}|^{2}-[Y-\hat {Y},Y-%
\hat{Y}]_{t}-2\int_{0}^{t}\langle X_{s-}-\hat{X}_{s-},dY_{s}-d\hat {Y}%
_{s}\rangle%
\end{align*}
which when combined with  (\ref{eq Lemma 4.3}) gives
\begin{equation*}
|X_{t}-\hat{X}_{t}|^{2}\leq|Y_{0}-\hat{Y}_{0}|^{2}+[Y-\hat{Y},Y-\hat{Y}%
]_{t}+2\int_{0}^{t}\langle X_{s-}-\hat{X}_{s-},dY_{s}-d\hat{Y}_{s}\rangle\,.
\end{equation*}
Therefore for any $p\in\mathbb{N}$ and any stopping time $\tau$
there exists $c_{p}>0$ such that%
\begin{align*}
\mathbb{E}\;{\|X-\hat{X}\|_{\tau}^{2p}}\leq c_{p}\,&\Big(\mathbb{
E}|Y_{0}-\hat{Y}_{0}|^{2p}+\mathbb{E}\left[ Y-\hat{Y}\right]
_{\tau}^{p}\\
&\qquad+\mathbb{E}
\sup_{t\leq\tau}\big|\int_{0}^{t}\langle X_{s-}-\hat {X}_{s-},dM_{s}-d\hat{M}%
_{s}\rangle\big|^{p} \\
&\qquad+\mathbb{E}\sup_{t\leq\tau}\big|\int_{0}^{t}\langle X_{s-}-\hat{
X}_{s-},dV_{s}-d\hat{V}_{s}\rangle\big|^{p}\Big).
\end{align*}
The rest of the proof runs as  the proof of  \cite[Theorem 1]{sl/94}.
\end{proof}
\medskip

We will need  the following conditions on  the coefficient $f$:
\begin{description}

\item[(H3)] $f:{\mathbb{R}}^{d}\rightarrow {{%
\mathbb{R}^{d}}}\otimes{{\mathbb{R}^{d}}}$ is a continuous
function such that%
\begin{equation*}
\|f(x)\|\leq L(1+|x|),\quad x\in{{\mathbb{R}^{d}}}.
\end{equation*}
\item[(H4)] For any $N\in{\mathbb{N}}$ there is $K_{N}>0$ such that%
\begin{equation*}
\|f(x)-f(y)\|\leq K_{N}|x-y|,\quad x,y\in B(0,N).
\end{equation*}
\end{description}

Let $H$ be an $(\mathcal{F}_{t})$-adapted process with trajectories in ${%
\mathbb{D}}\left( {\mathbb{R}^{+}},{{\mathbb{R}^{d}}}\right) $
such that $H_{0}\in \overline{\mathrm{D}(A)}$ and let $Z$ be an
$(\mathcal{F}_{t})$-adapted semimartingale such that $Z_{0}=0$.

\begin{definition}{\rm
\label{def4.4}We say that a pair $(X,K)$ of
$(\mathcal{F}_{t})$-adapted processes with trajectories  in
$\mathbb{D}\left( {\mathbb{R}^{+}},{{\mathbb{R}^{d}}}\right)$ is a
strong solution of  MSDE (\ref{eq1.1}) if
$(X,K)=\mathcal{SP}(A,\Pi;Y)$, $\mathbb{P}$-a.s.,  where
\begin{equation*}
Y_{t}=H_{t}+\int_{0}^{t}\langle f(X_{s-}),dZ_{s}\rangle,\;t\in\mathbb{R}%
^{+}.
\end{equation*}
}
\end{definition}

\begin{theorem}
\label{thm4.5} Assume (H1)--(H4). Let $H$ be an
$(\mathcal{F}_{t})$-adapted process with
trajectories in ${\mathbb{D}}\left( {\mathbb{R}^{+}},{{\mathbb{R}^{d}}}%
\right) $ such that $H_{0}\in\overline{\mathrm{D}(A)}$, and let $Z$ be an $%
(\mathcal{F}_{t})$-adapted semimartingale with $Z_{0}=0$. Then there exists a
unique strong solution $(X,K)$ of the MSDE (\ref{eq1.1}).
\end{theorem}

\begin{proof}
Without loss of  generality we may assume that
$|H_{t}|,|Z_{t}|\leq c$ for some constant $c>0$. Then $|\Delta
Z|\leq2c$, so by \cite[Chapter III Theorem 32]{pr/04}  $Z$ is a
special semimartingale  admitting a unique decomposition of the form $%
Z_{t}=M_{t}+V_{t}$, $t\in{\mathbb{R}^{+}}$, where $M$ is a local
square-integrable martingale with $|\Delta M|\leq4c$ and $V$ is a
predictable process with locally bounded variation with $|\Delta V|\leq2c$.

{\it Step 1.}  We first replace (H3), (H4) by the the following
stronger condition:
\begin{description}
\item[(H3*)] $f$ is a Lipschitz continuous function,
i.e. there exists  $L>0$ such that%
\begin{equation*}
\|f(x)-f(y)\|\leq L|x-y|,~x,y\in{{\mathbb{R}^{d}}}.
\end{equation*}
\end{description}
Let $L,C_{1}$ be the constants from (H3*) and Lemma \ref%
{lem4.3}, respectively.
Set%
\begin{equation*}
\tau^{\prime}=\inf\{t:|H_{t}+\langle f(H_{0}),Z_{t}\rangle-H_{0}|\geq
\ro/2\}\wedge1.
\end{equation*}
First we will show the existence and uniqueness of a solution of
the MSDE (\ref{eq1.1}) on the interval $[0,\tau)$, where $
\tau=\inf\{t>0:\max([M]_{t},\langle
M\rangle_{t},\updownarrow{\hspace{-0.1cm}
V\hspace{-0.1cm}}\updownarrow_{t}^{2})>b\}\wedge\tau^{\prime}$.
Set %
\begin{equation*}
\mathcal{S}^{2}=\{Y:Y\text{ is }\mathcal{F}_{t}\text{--adapted, }Y_{0}=H_{0}%
\text{, }Y_{t}=Y_{t}^{\tau-}\text{, }\mathbb{E}\sup_{t\geq0}|Y_{t}|^{2}<%
\infty\}
\end{equation*}
and define the mapping
$\Phi:\mathcal{S}^{2}\longrightarrow\mathcal{S}^{2}$ by putting
$\Phi(Y)$ to be the first coordinate of the solution of the
Skorokhod problem associated with
$H^{\tau-}+\int_{0}^{\cdot}\langle
f(Y_{s-}),dZ_{s}^{\tau-}\rangle$. We will show that $\Phi$ is a
contraction. Let us first observe that $
\Phi(H_{0})=H^{\tau-}+\langle f(H_{0}),Z^{\tau-}\rangle-K^{\tau-}$
and hence, by Corollary \ref{cor2.12},
$\Phi(H_{0})\in\mathcal{S}^{2}.$
By Lemma \ref{lem4.3}(ii), for any $Y,Y^{\prime}\in\mathcal{%
S}^{2}$ we have
\begin{align*}
\mathbb{E}\sup_{t<\tau}|\Phi(Y)_{t}-\Phi(Y^{\prime})_{t}|^{2}%
\leq&
C_{1}\{E\int_{0}^{\tau-}|f(Y_{s-})-f(Y_{s-}^{\prime})|^{2}\,d([M]_{s}+%
\langle M\rangle_{s}) \\
&\qquad+\mathbb{E}(\int_{0}^{\tau-}|f(Y_{s-})-f(Y_{s-}^{\prime
})|\,d\updownarrow{\hspace{-0.1cm}V\hspace{-0.1cm}}\updownarrow_{s})^{2}\}\medskip \\
\leq&3C_{1}bL^{2}\mathbb{E}%
\sup_{t<\tau}|Y_{t}-Y_{t}^{\prime}|^{2}=\frac{1}{2}\mathbb{E}%
\sup_{t<\tau}|Y_{t}-Y_{t}^{\prime}|^{2}.%
\end{align*}
Hence $\left[ \Phi(Y)-\Phi(H_{0})\right] \in\mathcal{S}^{2}$, and
consequently $\Phi(Y)\in S^{2}$; moreover we see that
$\Phi:S^{2}\rightarrow S^{2}$ is a contraction. Therefore  by the
Banach contraction principle
there exists a fixed point $X^{1}$. Note that $X^{1}$ is the first coordinate of the unique solution of (\ref%
{eq1.1}) on $[0,\tau)$. Indeed, set $X(0)=X_0$ and
\[(X(n),K(n))=\mathcal{SP}(A,\Pi;Y(n)),\]
where $Y(n)=H^{\tau-}+\int_{0}^{\cdot}\langle
f(X(n-1)_{s-}),dZ_{s}^{\tau-}\rangle)$, $n\in\N.$
Since $X(n)$ tends to $X^1$  in $S^{2}$, also $Y(n)$ tends to $Y^1=H^{\tau-}+\int_{0}^{\cdot}\langle
f(X^1_{s-}),dZ_{s}^{\tau-}\rangle$  in $S^{2}$, and by Remark \ref{rem2.6}(d), $(X(n),K(n))$ tends uniformly in probability to $(X^1,K^1)=\mathcal{SP}(A,\Pi;Y^1)$, which is a  unique solution of (\ref%
{eq1.1}) on $[0,\tau)$.
Moreover, putting
$X_{\tau}^{1}=\Pi(X_{\tau-}^{1}+ \Delta H_{\tau }+\langle
f(X_{\tau-}^{1}),\Delta Z_{\tau}\rangle)$ and  $K^1_{\tau}=X^1_{\tau}-H_{\tau}-\int_{0}^{\tau}\langle
f(X^1_{s-}),dZ_{s}\rangle$ we obtain a solution on
$[0,\tau]$.

Let us now define the sequence of stopping times
$\{\tau_{k}\}$ by putting $\tau _{1}=\tau$ and
\begin{equation*}
\tau_{k+1}=\tau_{k}+\inf\{t>0:\max([\hat{M}]_{t},\langle\hat{M}\rangle _{t},
\updownarrow{\hspace{-0.1cm}\hat{V}\hspace{-0.1cm}}
\updownarrow_{t}^{2})>b\}\wedge\tau_{k}^{\prime},\quad k\in{\mathbb{N}},
\end{equation*}
where $\hat{M}_{t}=M_{\tau_{k}+t}-M_{\tau_{k}}$, $\hat{V}_{t}=V_{%
\tau_{k}+t}-V_{\tau_{k}}$, $\tau_{k}^{\prime}=\inf\{t:|H_{\tau_{k}+t}+%
\langle
f(H_{\tau_{k}}),\hat{Z}_{t}\rangle-H_{\tau_{k}}|\geq\ro/2\}\wedge1
$. Arguing as above, we obtain a solution $(X^{k+1},K^{k+1})$ of
(\ref{eq1.1}) on $[\tau_{k},\tau_{k+1}]$. Since
$\tau_{k}\uparrow+\infty$, we get a solution $(X,K)$ on
${\mathbb{R}^{+}}$ by putting together the solutions $(X^{k+1},K^{k+1})$ on
$[\tau_{k},\tau_{k+1}]$, $k\in${$\mathbb{N}$}.

{\it Step 2.} The general case. For any $N\in{\mathbb{N}}$ there
exists a Lipschitz continuous function $f_{N}$ satisfying (H3*)
and such that $f_{N}(x)=f(x)$, $x\in B(0,N)$ and $f_{N}(x)=0$,
$x\in B^{c}(0,N+1)$. By Step 1, for any $N\in{\mathbb{N}}$ there
exists a unique strong solution of the equation
\begin{equation}
X_{t}^{N}+K_{t}^{N}=H_{t}+\int_{0}^{t}\langle
f_{N}(X_{s-}^{N}),\,dZ_{s}\rangle,\quad t\in{\mathbb{R}^{+}}.
\label{eq4.14}
\end{equation}
Set $\gamma_{0}=0$ and $
\gamma_{N}=\inf\{t:|X_{t}^{N}|>N\}$, $N\in{\mathbb{N}}$.
Since $f_{N}(x)=f_{N+1}(x)$ for $x\in B(0,N)$, $X_{t}^{N}=X_{t}^{N+1}$ for $%
t<\gamma_{N}$ and $\gamma_{N}\leq\gamma_{N+1}$. In order to finish the proof
it suffices  to show that%
\begin{equation}
\gamma_{N}\nearrow+\infty,\;\mathbb{P}\text{-a.s.}
\label{eq4.155}
\end{equation}
and  observe that the unique solution of (\ref{eq1.1}) has the
form $X_{t}=X_{t}^{N}$, $t\in\lbrack\gamma_{N-1},\gamma_{N})$,
$N\in{\mathbb{N}}$. Let $
(\wh{X},\wh{K})$ denote the solution of the Skorokhod problem with $%
\wh{Y}=H$. Set $
\beta_{k}=\inf\{t:|\wh{X}_{t}|
\vee\updownarrow{\hspace{-0.1cm}V\hspace{-0.1cm}}\updownarrow_{t}\vee\lbrack
M]_{t}\vee\langle M\rangle_{t}>k\}\wedge k$, $k\in{\mathbb{N}}$.
It is clear that%
\begin{equation}
\beta_{k}\nearrow+\infty,\;\mathbb{P}\text{-a.s.}   \label{eq4.16}
\end{equation}
By Lemma \ref{lem4.3}(ii) with $p=1$ and by  (H3), for every
stopping time $\sigma$,%
\begin{align*}
\mathbb{E}\sup_{t<\sigma\wedge\beta_{k}}&|X_{t}^{N}-\wh{X}_{t}|^{2}\leq C_{1}~\mathbb{E}{\Big[}\int_{0}^{(\sigma\wedge%
\beta_{k})-}\|f(X_{s-}^{N})\|^{2}\,d[M]_{s}\\
&\quad\qquad\quad+\int_{0}^{(\sigma\wedge%
\beta_{k})-}\|f(X_{s-}^{N})\|^{2}\,d\left\langle M\right\rangle
_{s}+k\int_{0}^{(\sigma\wedge\beta_{k})-}\|f(X_{s-}^{N})\|^{2}%
\,d\updownarrow{\hspace{-0.1cm}V\hspace{-0.1cm}}\updownarrow_{s}{\Big]} \\
&\quad\quad\leq C(k,L)~{\Big[}1+\mathbb{E}\int_{0}^{(\sigma\wedge%
\beta_{k})-}\sup_{u\leq s} |X_{u-}^{N}-\wh{X}_{u-}|^{2}\,
d(\updownarrow{\hspace{-0.1cm}V\hspace{-0.1cm}}\updownarrow
+[M]+\langle M\rangle
)_{s}{\Big].}%
\end{align*}
Therefore for every stopping time $\sigma$,%
\begin{align*}
\mathbb{E}
&\sup_{t<\sigma}|X_{t}^{N,\beta_{k}-}-\wh{X}_{t}^{\beta_{k}-}|^{2}\\
&\,\,\leq C(k,L)\Big(1+\mathbb{E}\int_{0}^{\sigma-}\sup_{u\leq
s}|X_{u-}^{N,\beta_{k}-}-\wh{X}_{u-}^{\beta_{k}-}|^{2}d(\updownarrow{\hspace{-0.1cm}V\hspace{-0.1cm}}\updownarrow^{\beta_{k}-}
+[M^{\beta_{k}-}]+\langle M^{\beta_{k}-}\rangle )_{s}\Big).
\end{align*}
By the above and Gronwall's lemma we obtain that
\[\mathbb{E}\sup_{t<\beta_{k}}|X_{t}^{N}-X_{t}^{\prime}|^{2}\leq
C(k,L)\exp \{3k\;C(k,L)\},\] which implies that for every
$k\in{\mathbb{N}}$,
$
\sup_{N}\mathbb{E}\sup_{t<\beta_{k}}|X_{t}^{N}|^{2}\leq C^{\prime}(k,L)$.
By this and Chebyshev's inequality,%
\begin{equation*}
\mathbb{P}(\gamma_{N}<\beta_{k})\leq P(\sup_{t<\beta_{k}}|X_{t}^{N}|\geq N)\leq%
\frac{C^{\prime}(k,L)}{N^{2}},
\end{equation*}
which converges to zero as $N\rightarrow\infty$. Therefore
(\ref{eq4.16}) implies (\ref{eq4.155}), which competes the proof.
\end{proof}
\medskip


Now we are going to study approximations of solutions of the MSDE
(\ref{eq1.1})  under the assumptions (H1)--(H4). First we
consider discrete approximation schemes, which are
constructed with the natural analogy to the Euler scheme. Let $%
\{\pi_{n}=\{0=t_{n0}<t_{n1}<\dots<t_{nk}<\dots\}\}$  be a sequence
of partitions of $\mathbb{R}^{+}$ such that  $\lim
_{n\rightarrow\infty}\max(t_{nk}-t_{n,k-1} )=0$.

Set%
\begin{equation}
\bar{X}_{t}^{n}=\left\{
\begin{array}{ll}
\mathcal{SP}^{\left( 1\right) }(A;H_{0})_{t}, & t\in\lbrack0,t_{n,1}),%
\medskip \\[2mm]
\mathcal{SP}^{\left( 1\right) }\big(A;\Pi(\bar{X}%
_{t_{n,k-1}}^{n}+(H_{t_{nk}}-H_{t_{n,k-1}}) & \medskip \\
\quad+\langle f(\bar{X}_{t_{n,k-1}}^{n}),(Z_{t_{nk}}-Z_{t_{n,k-1}})%
\rangle)\big)_{t-t_{n,k}}, & t\in\lbrack t_{nk},t_{n,k+1}),\,k\in {\mathbb{N}%
}.\label{eq33}
\end{array}
\right.
\end{equation}
Let $(\mathcal{F}_{t}^{n})_{t\geq0} $ denote the discretization of $
(\mathcal{F}_{t})_{t\geq0}$, i.e. $\mathcal{F}_{t}^{n}=\mathcal{F}_{t_{nk}}$ for $%
t\in\lbrack t_{nk},t_{n,k+1})$, and let $H_{t}^{(n)}=H_{t_{nk}}$, $%
Z_{t}^{(n)}=Z_{t_{nk}}$ for $t\in\lbrack t_{nk},t_{n,k+1})$,
$k\in{\mathbb{N} }$, $n\in{\mathbb{N}}$.
Set%
\begin{equation*}
\bar{Y}_{t}^{n}=H_{t}^{(n)}+\int_{0}^{t}\langle f(\bar{X}_{s-}^{n}),%
\,dZ_{s}^{(n)}\rangle,\quad t\in{\mathbb{R}^{+}},\,n\in{\mathbb{N}}
\end{equation*}
and note that $H^{(n)},Z^{(n)}$ and $\bar{X}^{n}$, $\bar{K}^{n}=\bar{X}^{n}-%
\bar{Y}^{n}$ are $(\mathcal{F}_{t}^{n})$-adapted processes such that $(\bar{X}%
^{n},\bar{K}^{n})=\mathcal{SP}(A,\Pi;\bar{Y}^{n})$, $n\in{\mathbb{N}}$.%

\begin{theorem}
\label{thm4.6}Under  assumptions (H1)--(H4),
\begin{enumerate}
\item[\rm(i)] $\displaystyle{
{(\bar{X}^{n},\bar{K}^{n},H^{(n)},Z^{(n)})\xrightarrow[\mathcal{P}]{\;\;\;\;
\;\;\;\;}(X,K,H,Z)\quad}\text{in }{{\mathbb{D}}}\left( {{\mathbb{R}^{+}},%
\mathbb{R}^{4d}}\right)}$,
\item[\rm(ii)]  for every $T\in{\mathbb{R}^{+}}$%
\begin{equation*}
\sup_{t\leq T,\,t\in\pi_{n}}|\bar{X}_{t}^{n}-X_{t}|{\xrightarrow[\mathcal{P}]{%
\;\;\;\; \;\;\;\;}}0\quad\text{and}\quad \sup_{t\leq T,\,t\in\pi_{n}}|\bar{K}%
_{t}^{n}-K_{t}|{\xrightarrow[\mathcal{P}]{\;\;\;\; \;\;\;\;}}0\,,
\end{equation*}
\item[\rm(iii)]  for every $t\in{\mathbb{R}^{+}}$ such that $%
\mathbb{P}(\Delta H_{t}=\Delta Z_{t}=0)=1$ or $t\in\liminf_{n\rightarrow
+\infty}\pi_{n}$,%
\begin{equation*}
\bar{X}_{t}^{n}{\xrightarrow[\mathcal{P}]{\;\;\;\; \;\;\;\;}}X_{t}\quad\text{%
and}\quad\bar{K}_{t}^{n}{\xrightarrow[\mathcal{P}]{\;\;\;\; \;\;\;\;}}%
K_{t}\,,
\end{equation*}
\end{enumerate}
where $(X,K)$ is a strong solution of the SDE (\ref{eq1.1}).
\end{theorem}

\begin{proof} (i) Set
$Y_{t}=H_{t}+\int_{0}^{t}\langle f(X_{s-}),Z_{s}\rangle$, $t\in{\mathbb{R}^{+}}$
and $Y_{t}^{(n)}=Y_{t_{nk}}\,$ for $t\in\lbrack t_{nk},t_{n,k+1})$, $k\in{%
\mathbb{N}}$, $n\in{\mathbb{N}}$. Let
$(X^{n},K^{n})=\mathcal{SP}(A,\Pi;Y^{(n)})$, $n\in{\mathbb{N}}$.
By Remark \ref{rem2.6}(d) and  arguments from the proof of
Proposition \ref{prop1},
\begin{equation}
(X^{n},K^{n},H^{(n)},Z^{(n)})\longrightarrow(X,K,H,Z),~~\mathbb{P}\text{%
-a.s. in }{\mathbb{D}}({\mathbb{R}^{+},}\mathbb{R}^{4d}).
\label{eq4.3}
\end{equation}
Let $(\hat{X}^{n},\hat{K}^{n})=\mathcal{SP}(A,\Pi;\hat{Y}^{n})$,
where%
\begin{equation*}
\hat{Y}_{t}^{n}=H_{t}^{(n)}+\int_{0}^{t}\langle
f(X_{s-}^{n}),dZ_{s}^{(n)}\rangle,\hspace{0.2cm}t\in{\mathbb{R}^{+}},~n\in{%
\mathbb{N}}.
\end{equation*}
By (\ref{eq4.3}) and the theorem on the functional convergence of
stochastic integrals (see, e.g., \cite[Theorem 2.11]{ja-me-pa/89}), \[
(\hat{Y}^{n},X^{n},K^n,H^{(n)},Z^{(n)}){\xrightarrow[\mathcal{P}]{\;\;\;\;
\;\;\;\;}}(Y,X,K,H,Z)\text{~in }{\mathbb{D}}({\mathbb{R}^{+},}\mathbb{R}^{5d}%
\mathbb{)}.\]
 Therefore using once again Remark \ref{rem2.6}(d) we
get
\begin{equation}
(\hat{X}^{n},\hat{K}^{n},X^{n},K^n,H^{(n)},Z^{(n)}){\xrightarrow[\mathcal{P}]{\;%
\;\;\; \;\;\;\;}}(X,K,X,K,H,Z)\text{~in }{\mathbb{D}}({\mathbb{R}^{+},}\mathbb{%
R}^{6d}\mathbb{)},   \label{eq4.4}
\end{equation}
which implies in particular  that
$\|X^{n}-\hat{X}^{n}\|_{T}\xrightarrow[\mathcal{P}]{\;\;\;\;\;}0$
and
$\|K^{n}-\hat{K}^{n}\|_{T}\xrightarrow[\mathcal{P}]{\;\;\;\;\;}0$
for $T\in{\mathbb{R}^{+}}$. In order to complete the proof of (i)
it is sufficient to show that for any $T\in{\mathbb{R}^{+}}$,%
\begin{equation}
\|\bar{X}^{n}-\hat{X}^{n}\|_{T}{\xrightarrow[\mathcal{P}]{\;\;\;\;
\;\;\;\;}}0\quad\text{and}\quad\|\bar {K}^{n}-\hat{K}^{n}\|_{T}{%
\xrightarrow[\mathcal{P}]{\;\;\;\; \;\;\;\;}}0.   \label{eq4.5}
\end{equation}
Under  assumption (H3*)
the proof of (\ref{eq4.5}) runs as the proof of (3.8) in \cite[Theorem 3.5]{sl-wo/10}. To prove the general case we set
$\gamma_{N}=\inf\{t:|X_{t}|>N\}$, $N\in{\mathbb{N}%
}$. The arguments used previously show the convergence of
approximating sequences on the sets $\{T\leq\gamma_{N}\}$. Since
$\gamma_{N}\nearrow+\infty$, $\mathbb{P}$-a.s., the result
follows.

(ii) Set $X_{t}^{(n)}=X_{t_{nk}}\,$, $K_{t}^{(n)}=K_{t_{nk}}$ for
$t\in\lbrack t_{nk},t_{n,k+1})$, $k\in{%
\mathbb{N}}$, $n\in{\mathbb{N}}$. From (i) we deduce that%
\begin{equation*}
(\bar X^{n},X^{(n)},\bar K^{n},K^{(n)})\xrightarrow[\mathcal{P}]{\;\;\;\;\;} (X,X,K,K)\quad\text{in }{\mathbb{%
D}}({\mathbb{R}^{+}},{\mathbb{R}}^{4}).
\end{equation*}
By the above  and \cite[Chapter VI Proposition 1.17]{ja-sh/87},
\begin{equation*}
\sup_{t\leq
T,t\in\pi_{n}}|X_{t}^{n}-X_{t}|=\|\bar X^{n}-X^{(n)}\|_{T} \xrightarrow[\mathcal{P}]{\;\;\;\;\;}  0,\,T%
\in{\mathbb{R}^{+}}
\end{equation*}
and%
\begin{equation*}
\sup_{t\leq T,\,t\in\pi_{n}}|K_{t}^{n}-K_{t}|=\|\bar{K}^{n}-K^{(n)}\|_{T}%
\xrightarrow[\mathcal{P}] {\;\;\;\;\;}0,\,T\in{\mathbb{R}^{+}},
\end{equation*}
which completes the proof.

(iii) Follows easily from (i) and (ii).
\end{proof}
\medskip

For $n\in{\mathbb{N}}$ and $z\in{{\mathbb{R}^{d}}}$ set
\begin{equation*}
J_{n}(z)=\big(I+\frac{A}{n}\big)^{-1}(z),\,\,\,\,A_{n}(z)=n(z-J_{n}(z)).
\end{equation*}
($A_{n}$ is called the Yosida approximation of the operator $A$).

\begin{remark}{\rm
\label{rem3.1}It is well known (see, e.g., \cite{br/73}) that
$A_{n}$ is a maximal monotone operator such that for all
$z,z^{\prime}\in\mathbb{R}^{d}$ and $n\in\mathbb{N}$,
\begin{enumerate}
\item[(a)]
$|J_{n}(z)-J_{n}(z^{\prime})|\leq|z-z^{\prime}|$,
\item[(b)] $|A_{n}(z)-A_{n}(z^{\prime})|\leq n|z-z^{\prime}|$,
\item[(c)] $\lim_{n\rightarrow\infty}J_{n}(z)=\Pi_{\overline{D(A)}}(z)$,
\item[(d)] $\displaystyle\langle z-z^{\prime},A_{n}(z)-A_{n}(z^{\prime})\rangle
\geq\frac{1}{n}(|A_{n}(z)|^{2}+|A_{n}(z^{\prime})|^2-2\langle
A_{n}(z),A_{n}(z^{\prime})\rangle)\geq0$.
\end{enumerate} }
\end{remark}

Since $A_{n}$ is Lipschitz continuous,  there is a
unique solution $X^{n}=\mathcal{SP}^{\left( 1\right) }(A_{n},\Pi _{\overline{%
\mathrm{D}\left( A_{n}\right) }};Y)$. We  call $X^{n}$ the
solution of
the Yosida problem and denote it  by $X^{n}=\mathcal{YP}%
(A_{n};Y)$), $n\in{\mathbb{N}}$. We remark that in fact, $\mathcal{YP%
}(A_{n};Y)=\mathcal{SP}^{\left( 1\right) }(A_{n};Y)$, because the domain of $%
A_{n}$ is $\mathbb{R}^{d}$ and the generalized projection $\Pi_{\overline{%
\mathrm{D}\left( A_{n}\right) }}$ becomes the identity.

\begin{lemma}
\label{lem4.7}Assume (H1), (H2). Let $Y,\hat{Y}$ be two processes admitting decompositions (\ref%
{eq4.1}) and  let $X^{n}=\mathcal{YP}(A_{n};Y)$, $\hat{X}^{n}=\mathcal{YP}%
(A_{n};\hat{Y})$, $n\in\mathbb{N}$. Then for any $p\in
{\mathbb{N}}$ there exists a constant $C_{p}>0$ such that for any
stopping time $\tau$ and $n\in\mathbb{N}$,%
\begin{enumerate}
\item[\rm(i)] $\mathbb{E}{\|X^{n}-\hat{X}^{n}\|_{\tau}^{2p}
\leq C_{p}\mathbb{E}}\big( {|Y_{0}-\hat{Y}_{0}|^{2p}+[M-\hat {M}%
]_{\tau}^{p}+\updownarrow{\hspace{-0.1cm}V-\hat{V}\hspace{-0.1cm}}\updownarrow_{\tau}^{2p}}\big) $,
\item[\rm(ii)]$\!\!\!\!\mathbb{E}{\|X^{n}-\hat {X}^{n}\|^{2p}_{\tau-}\leq C_{p}\mathbb{E}}\big({|Y_{0}-\hat{Y}_{0}|^{2p}+[M-%
\hat{M}]_{\tau-}^{p}+\langle M-\hat{M}\rangle_{\tau-}^{p}+\updownarrow{\hspace{-0.1cm}V-\hat{V}\hspace{-0.1cm}}\updownarrow%
_{\tau-}^{2p}}\big)$.
\end{enumerate}
\end{lemma}
\begin{proof}
Set $K^{n}=Y-X^{n},\,\hat{K}^{n}=\hat{Y}-\hat{X}^{n}$, $%
n\in\mathbb{N}$. By \cite[Lemma 29]{ma-ra-sl/13},
\begin{equation*}
|X_{t}^{n}-\hat{X}_{t}^{n}|^{2}\leq|Y_{t}-\hat{Y}_{t}|^{2}-2\int_{0}^{t}%
\langle Y_{t}-\hat{Y}_{t}-Y_{s}+\hat{Y}_{s},dK_{s}^{n}-d\hat{K}%
_{s}^{n}\rangle,\hspace{0.2cm}t\in{\mathbb{R}^{+}}.
\end{equation*}
The rest of the proof runs as  the proof of Lemma \ref{lem4.3}.
\end{proof}

\begin{theorem}
\label{thm4.8} Assume (H1)--(H4) and denote by  $X^{n}$  the
solution of (\ref{eq1.3}).
\begin{enumerate}
\item[\rm(i)]
For any stopping time $\tau$ such that $\mathbb{P%
}(\tau<+\infty)=1,$%
\begin{equation*}
X_{\tau}^{n}{\xrightarrow[\mathcal{P}]{\;\;\;\;
\;\;\;\;}}\bar{X}_{\tau }=X_{\tau-}+\Delta H_{\tau}+\langle
f(X_{\tau-}),\Delta Z_{\tau}\rangle.
\end{equation*}
In particular, $X_{\tau}^{n}\xrightarrow[\mathcal{P}]{\;\;\;\;\;}%
X_{\tau }$ provided that $\mathbb{P}(\Delta H_{\tau}=\Delta
Z_{\tau}=0)=1$.
\item[\rm(ii)]
For any $T\in{\mathbb{R}^{+}}$,%
\begin{equation*}
\|J_{n}(X^{n})-X\|_{T}{\xrightarrow[\mathcal{P}]{\;\;\;\;
\;\;\;\;}}0,
\end{equation*}
\end{enumerate}
where $(X,K)$ is a strong solution of the MSDE (\ref{eq1.1}) with $\Pi=\Pi _{%
\overline{\mathrm{D}(A)}}$ .
\end{theorem}

\begin{proof} (i) Set
$Y_{t}=H_{t}+\int_{0}^{t}\langle f(X_{s-}),dZ_{s}\rangle$, $t\in{\mathbb{R}^{+}
}$.
Let $\hat{X}^{n}=\mathcal{YP}(A_{n};Y)$, $n\in\mathbb{N}$. By
\cite[Theorem 31(j),(jj)]{ma-ra-sl/13},
for any stopping time $\tau$ such that $\mathbb{P}(\tau<+\infty)=1$,%
\begin{equation}
\hat{X}_{\tau}^{n}\longrightarrow\bar{X}%
_{\tau }=X_{\tau-}+\Delta H_{\tau}+\langle f(X_{\tau-}),\Delta
Z_{\tau}\rangle, \quad\mathbb{P}\text{%
-a.s.},  \label{eq4.6}
\end{equation}
and for every $t\in{%
\mathbb{R}^{+}}$,
\begin{equation}
\hat{X}_{t-}^{n}\longrightarrow X_{t-},\quad \quad\mathbb{P}\text{%
-a.s.  }   \label{eq4.7}
\end{equation}
By (\ref{eq4.6}), to prove (i) it suffices to show that
\begin{equation}
\|\hat{X}^{n}-X^{n}\|_{T}{\xrightarrow[\mathcal{P}]{\;\;\;\; \;\;\;\;}}%
0,~T\in{\mathbb{R}^{+}}.   \label{eq4.8}
\end{equation}
Without loss of generality we may assume
that there is a constant $c>0$ such that $|Z_{t}|\leq c$. Then $Z$
is a special semimartingale admitting the decomposition $Z=M+V$,
where $M$ is a local square-integrable martingale such that
$M_0=0$ and $|\Delta M|\leq4c$ and $V$ is a predictable process of
locally bounded variation such that  $|\Delta V|\leq2c$ and
$V_{0}=0$. For $b>0$ set
\begin{equation*}
\tau_{n}^{b}=\inf\{t>0:\max([M]_{t},\langle M\rangle_{t},
\updownarrow{\hspace{-0.1cm}V\hspace{-0.1cm}}\updownarrow_{t},|\hat{X}%
_{t}^{n}|,|X_{t}|)>b\},\quad n\in\N.
\end{equation*}
By \cite[Theorem 32(j)]{ma-ra-sl/13}, for any $T\in{\mathbb{R}%
^{+}}$, the family $\{\|\hat{X}^{n}\|_{T}\}$ is bounded in
probability,
which implies that%
\begin{equation}
\lim_{b\rightarrow\infty}\limsup_{n\rightarrow\infty}\mathbb{P}%
(\tau_{n}^{b}\leq T)=0,~T\in{\mathbb{R}^{+}}.   \label{eq4.9}
\end{equation}
As in the proof of Theorem \ref{thm4.6}, without loss of
generality we can replace the assumptions (H3), (H4) by  (H3*).
Then by Lemma \ref{lem4.7}(ii) with $p=1$, for any stopping time
$\sigma_{n}$ we have
\begin{align*}
\mathbb{E}&\sup_{t<\sigma_{n}\wedge\tau_{n}^{b}}|X_{t}^{n}-\hat {
X}_{t}^{n}|^{2}\leq
C_{1}\Big[\mathbb{E}\int_{0}^{(\sigma_{n}\wedge\tau
_{n}^{b})-}\|f(X_{s-}^{n})-f(X_{s-})\|^{2}d([M]_{s}+\langle
M\rangle_{s})\\
&\qquad\qquad\qquad\qquad\qquad\qquad\quad+b\mathbb{E}\int_{0}^{(\sigma_{n}\wedge%
\tau_{n}^{b})-}\|f(X_{s-}^{n})-f(X_{s-})\|^{2}
d\updownarrow{\hspace{-0.1cm}V\hspace{-0.1cm}}\updownarrow_{s}\Big]\\
&\quad\leq2C_{1}L^{2}\Big[\mathbb{E}\int_{0}^{(\sigma
_{n}\wedge\tau_{n}^{b})-}\sup_{u\leq s}|X_{u-}^{n}-\hat{X}%
_{u-}^{n}|^{2}d([M]_{s}+<M>_{s}+b\updownarrow{\hspace{-0.1cm}V\hspace{-0.1cm}}\updownarrow_{s})+\epsilon_{n}\Big],%
\end{align*}
where
$\epsilon_{n}=\mathbb{E}\int_{0}^{(\sigma_{n}\wedge\tau_{n}^{b})-}|\hat{X}%
_{s-}^{n}-X_{s-}|^{2}\,d([M]_{s}+<M>_{s}+b\updownarrow{\hspace{-0.1cm}V\hspace{-0.1cm}}\updownarrow_{s}),~n\in\mathbb{N}$.
Therefore applying  Gronwall's lemma we obtain
\begin{equation*}
\mathbb{E}\sup_{t<\tau_{n}^{b}}|X_{t}^{n}-\hat{X}_{t}^{n}|^{2}%
\leq2C_{1}L^{2}\epsilon_{n}\exp\{2C_{1}L^{2}(2b+b^{2})\}.
\end{equation*}
Since (\ref{eq4.7}) implies that $\epsilon_{n}\rightarrow0$, (\ref{eq4.8}%
) follows from (\ref{eq4.9}). This completes the proof of
(i).\smallskip\\
(ii) By \cite[Theorem 31 (jj)]{ma-ra-sl/13}, for every
$T\in{\mathbb{R}^{+}}$,
\begin{equation}
\label{eq3.14} \|J_{n}(\hat{X}^{n})-X\|_{T}\longrightarrow 0,\quad
\mathbb{P}\text{-a.s.  }
\end{equation}
By the Lipschitz property of the operator $J_{n}$,
\begin{equation*}
\|J_{n}(X^{n})-X\|_{T}\leq\|J_{n}(\hat{X}^{n})-X\|_{T}
+\|\hat{X}^{n}-X^{n}\|_{T}\;,\quad
T\in{\mathbb{R}^{+}}.
\end{equation*}
When combined  with (\ref{eq4.8}),  (\ref{eq3.14}) this proves (ii).
\end{proof}
\medskip

Let  $Y$ be an $(\mathcal{F}_{t})$-adapted c\`{a}dl\`{a}g process
such that $Y_{0}\in\overline{\mathrm{D}\left( A\right) }$. Note
that for every $n\in{\mathbb{N}}$ there exists a unique
$(\mathcal{F}_{t})$-adapted c\`{a}dl\`{a}g process $X^n$ satisfying
the  equation
\begin{equation}\label{eqf}X_{t}^{n}
+K_{t}^{n}=Y_{t},\quad t\in{\mathbb{R}^{+}},
\end{equation}
where $
K_{t}^{n}=\int_{0}^{t}A_{n}(X_{s}^{n})ds-\sum_{s\leq t}\big[%
X_{s-}^{n}+\Delta Y_{s}-\Pi(X_{s-}^{n}+\Delta Y_{s})\big]\;\mbox{\bf 1}%
_{\{|\Delta Y_{s}|>1/n\}}$,  $t\in{\mathbb{R}^{+}}$. Indeed, if we
set $\sigma_{0}=0$ and $\sigma_{k+1}=\inf\{t>\sigma_{k}:|\Delta
Y_{t}|>1/n\}$, $k\in{\mathbb{N}}$, then on every stochastic
interval $[\sigma_{k},\sigma_{k+1})$, $X^{n}$ is a solution of the
equation  of the form
\[
X_{t}^{n}=\Pi(X_{\sigma_{k}-}^{n}+\Delta
Y_{\sigma_{k}})+Y_{t}-Y_{\sigma_{k}}
-\int_{\sigma_{k}}^{t}A_{n}(X_{s}^{n})ds,\quad
t\in[\sigma_{k},\sigma_{k+1})
\]
with Lipschitz continuous $A_n$. Since $\mathbb{P}$-a.s. there
exists only a finite number of jumps of $Y$ greater than $1/n$,
the process $X^n$ is well defined.
In the sequel we  will use the notation $X^{n}=\mathcal{YP}%
(A_{n},\Pi;Y)$, $n\in{\mathbb{N}}$.

Let $Y,\hat{Y}$ be processes admitting decompositions
(\ref{eq4.1}). Using  \cite[Proposition 35]{ma-ra-sl/13}  and
arguing as in the proof of Lemma \ref{lem4.7} we obtain the
following result.

\begin{lemma}Assume (H1), (H2).
\label{lem4.9}Let $X^{n}=\mathcal{YP}(A_{n},\Pi;Y)$ and $\hat{X}%
^{n}=\mathcal{YP}(A_{n},\Pi,\hat{Y})$, $n\in\mathbb{N}$. Then for every
$p\in{\mathbb{N}}$ there exists a constant $C_{p}>0$ such that for
any stopping time $%
\tau$ and  $n\in\mathbb{N}$ the estimates (i), (ii) from Lemma \ref%
{lem4.7} hold true.
\end{lemma}

In the rest of  Section 3 we consider approximations of
(\ref{eq1.1})  of the form%
\begin{equation}
X_{t}^{n}+K_{t}^{n}=H_{t}+\int_{0}^{t}
\langle f(X_{s-}^{n}),dZ_{s}\rangle,
\label{eq4.10}
\end{equation}
where
\[K_{t}^{n}=\int_{0}^{t}A_{n}(X_{s}^{n})ds-\sum_{s\leq t}\big[%
X_{s-}^{n}+\Delta Y_{s}^{n}-\Pi(X_{s-}^{n}
+\Delta Y_{s}^{n})\big]\;\mbox{\bf 1}%
_{\{\max(|\Delta H_{s}|,|\Delta Z_{s}|>1/n\}}\] and $
Y_{t}^{n}=H_{t}+\int_{0}^{t}\langle f(X_{s-}^{n}),dZ_{s}\rangle$, $t\in{%
\mathbb{R}^{+}}$, $n\in\N$. Note that there exists a unique
$(\mathcal{F}_{t})$-adapted solution of (\ref{eq4.10}). To check
this set  $\gamma_{0}=0$, $
\gamma_{k+1}=\inf\{t>\gamma_{k}:\max(|\Delta H_{t}|,|\Delta
Z_{t}|)>1/n\}$, $k\in{\mathbb{N}}$,  and observe that $X^{n}$
is a solution of the equation%
\begin{equation*}
\begin{array}{r}
\displaystyle X_{t}^{n}=\Pi(X_{\gamma_{k}-}^{n}+\Delta
H_{\gamma_{k}}+\langle f(X_{\gamma_{k}-}^{n}),\Delta
Z_{\gamma_{k}}\rangle)+H_{t}-H_{\gamma_{k}}\medskip \\
\displaystyle+\int_{\gamma_{k}}^{t}\langle f(X_{s-}^{n}),dZ_{s}\rangle
-\int_{\gamma_{k}}^{t}A_{n}(X_{s}^{n})ds%
\end{array}
\end{equation*}
on each  stochastic interval $[\gamma_{k},\gamma_{k+1})$. Since
$f$ satisfies  (H3), (H4)  and  $A_{n}$ is Lipschitz continuous,
it is known that the above equation has a unique strong solution
on $[\gamma_{k},\gamma_{k+1})$ for every $k\in\mathbb{N}$.
Therefore  (\ref{eq4.10})  has  a unique strong solution.

\begin{theorem}
\label{thm4.10} Assume (H1)--(H4) and denote by   $X^{n}$  the
solution of (\ref{eq4.10}). Then
\begin{equation*}
\|X^{n}-X\|_{T}{\xrightarrow[\mathcal{P}]{\;\;\;\; \;\;\;\;}}0,~T\in {%
\mathbb{R}^{+}},
\end{equation*}
where $(X,K)$ is the unique strong solution of the MSDE
(\ref{eq1.1}).
\end{theorem}

\begin{proof}
Let%
\begin{equation*}
\hat{X}^{n}=\mathcal{YP}(A_{n},\Pi;Y),\mathbb{\;}n\in\mathbb{N},
\end{equation*}
where $
Y_{t}=H_{t}+\int_{0}^{t}\langle f(X_{s-}),dZ_{s}\rangle$, $t\in{\mathbb{R}^{+}%
}$. By \cite[Theorem 36(j)]{ma-ra-sl/13}, for every
$T\in\mathbb{R}^{+}$,
\begin{equation*}
\|\hat{X}^{n}-X\|_{T}\longrightarrow0,~\mathbb{P}\text{--a.s.}
\end{equation*}
On the other hand, similarly to the proof of (\ref{eq4.8}) (using
Lemma \ref{lem4.9}(ii) instead of Lemma \ref{lem4.7}(ii)) one can
show that
\begin{equation*}
\|\hat{X}^{n}-X^n\|_{T}{\xrightarrow [\mathcal{P}]{\;\;\;\;
\;\;\;\;}}0
\end{equation*}
for $T\in\mathbb{R}^{+}$, which completes the proof.
\end{proof}

\section{Stability of MSDEs with maximal monotone operators}

For $n\in\N$ let $Z^n$  be an ${\cal F}^n_t$-adapted
semimartingale.  We will assume that $\{Z^n\}$ satisfies the
following condition  (UT) introduced in Stricker \cite{st}:
\begin{description}
\item[{(UT)}]\index{Condition (UT)} For every  $T\in\Rp$ the
family of random variables
\[
\{\int_{[0,T]} U^n_s\,dZ^n_s;\,n\in\N\,,\,U^n\in\mbox{\bf U}^n_T\}
\]
is bounded in probability. Here $ \mbox{\bf U}^n_T$  is the class
of all discrete predictable processes of the form
$U^n_s=U^n_0+\sum_{i=0}^kU^n_i\mbox{\bf 1}_{\{t_i<s\leq
t_{i+1}\}}$, where  $0=t_0<t_1<\dots<t_k=T$,  $U^n_i$  is ${\cal
F}^n_{t_i}$-measurable and $|U^n_i|\leq 1$  for
$i\in\{0,\dots,k\},n,k\in\N.$
\end{description}

\begin{remark}{\rm \label{rem4.1} A simple characterization of
(UT)  is given in \cite{ms}. To formulate it, let us first recall
that for every $a>0$ the semimartingale $Z^n$ admits decomposition
of the form
\begin{equation}
\label{eq0.3} Z^n=J^{n,a}+M^{n,a}+B^{n,a},
\end{equation}
where $\,J^{n,a}_t=\sum_{0<s\leq t}\Delta Z^n_s\mbox{\bf
1}_{\{|\Delta Z^n_s|>a\,\}}\,,\,M^{n,a}$\glossary{$|x|$}  is a
locally square integrable martingale\index{martingale}  with
$M^{n,a}_0=0$  and  $B^{n,a}$  is a predictable
process\index{predictable process}  of bounded
variation\index{process  of bounded variation} with
$\,B^{n,a}_0=0$. Theorem 1.4. in \cite{ms} asserts that $\{Z^n\}$
satisfies  the condition  (UT) if and only if for some $a>0$ and
for every  $T\in\Rp$  the families of random variables
$\{\updownarrow{\hspace{-0.1cm}J^{n,a}\hspace{-0.1cm}}\updownarrow_{T}\}\,$,
$\,\{\updownarrow{\hspace{-0.1cm}B^{n,a}\hspace{-0.1cm}}\updownarrow_{T}\}\,$,
$\,\{[M^{n,a}]_T\}$  are bounded in probability. }
\end{remark}
The condition (UT)  proved to be very useful in the theory of
limit theorems for stochastic integrals  and for solutions of SDEs
(see, e.g., \cite{ja-me-pa/89,kp,ms,s1,sl/93}).
\begin{lemma}
\label{lem4.1} Assume (H1), (H2) and  that $\{Y^n\}$ is a sequences of
$({\cal F}^n_t)$-adapted processes of the form
\[
Y^n=H^n+Z^n,\quad n\in\N,
\]
where $\{H^n\}$ is a tight in $\D$ sequence of $({\cal
F}^n_t)$-adapted processes with $H^n_0\in\ov{D(A)}$ and $\{Z^n\}$
is a sequence of $({\cal F}^n_t)$-adapted semimartingales with
$Z^n_0=0$ satisfying (UT). Let
$\{(X^n,K^n)=\mathcal{SP}(A,\Pi;Y^n)\}$ be a sequence of solutions
of the Skorokhod problem. Then for every $T\in\Rp$ the sequences
$\{\|X^n\|_T\}$  and
$\{\updownarrow{\hspace{-0.1cm}K^n\hspace{-0.1cm}}\updownarrow_{T}\}$
 are bounded in probability.
\end{lemma}
\begin{proof} By Proposition \ref{prop2.11} it suffices  to check
that for every $T\in\Rp$,
\begin{equation}\label{eq4.01}
\{\|Y^n\|_T\}\quad\mbox{\rm is bounded in probability}
\end{equation}
and
\begin{equation}\label{eq4.02}
\{N_{\ro/2}(Y^n,T)\}\quad\mbox{\rm is bounded in probability.}
\end{equation}
Since $\{Z^n\}$ satisfies (UT), it follows from \cite[Theorem 3.4.1]{ja} that  $\{Z^n\}$ is $S$-tight.
 By this and \cite[Theorem 3.3.3]{ja} for every $T\in\Rp$  and $\eta>0$,
\begin{equation}\label{eq4.03}
\{\|Z^n\|_T\},\quad\{N_{\eta}(Z^n,T)\}\quad\mbox{\rm are bounded
in probability.}
\end{equation}
On the other hand, by tightness of $\{H^n\}$  in $\D$,  for every
$T\in\Rp$  and $\eta>0$,
\begin{equation}\label{eq4.04}
\{\|H^n\|_T\}\quad\mbox{\rm is bounded in probability}
\end{equation}
and
\begin{equation}\label{eq4.05}
\lim_{\delta\to0}\sup_nP(\omega'_{\delta}(H^n,T)\geq
\eta)=0.\end{equation} Clearly, (\ref{eq4.03}) and (\ref{eq4.04})
imply (\ref{eq4.01}). In order to check (\ref{eq4.02}) and
complete the proof we will show that (\ref{eq4.05}) implies that
for every $T\in\Rp$  and $\eta>0$ the sequence
$\{N_{\eta}(H^n,T)\}$ is bounded in probability. Let $\epsilon>0$.
By (\ref{eq4.05}) there is $\delta_\epsilon>0$ such that
\begin{equation}\label{eq4.06}
\sup_nP(\omega'_{\delta_\epsilon}(H^n,T)\geq
\eta)\leq\epsilon.
\end{equation}
Observe that if
$\omega'_{\delta_\epsilon}(H^n(\omega),T)<\eta$ for some
$\omega\in\Omega$, then there exists a subdivision $(s_k)$ of
$[0,T]$ such that $0=s_0 < s_1 < \dots <s_m = T$ ,
$\delta_\epsilon \leq s_k - s_{k-1}$ , $k=1, \dots ,m-1$, where
$m=[T/\delta_\epsilon ]+1$, and
$\omega_{H^n(\omega)}([s_{k-1},s_k))<\eta$. Hence, in particular,
$N_\eta(H^n(\omega),T)\leq m$. Consequently, (\ref{eq4.06})
implies that for every $\epsilon>0$ there is
$K_\epsilon=[T/\delta_\epsilon ]+1$ such that
\[
\sup_nP(N_\eta(H^n,T)>K_\epsilon)\leq\epsilon,
\]
which completes the proof of (\ref{eq4.02}).
\end{proof}

\begin{corollary} Assume (H1), (H2).
\label{cor2.8} For  $n,i\in\N$ let $Y^{ni}$  and $\wh Y^{ni}$  be
 processes adapted to filtrations
$({\cal F}_t^{ni})_{t\geq0}$ and $(\wh {\cal F}_t^{ni})_{t\geq0}$, respectively, and let
$(X^{ni},K^{ni})= \mathcal{SP}(A,\Pi;Y^{ni})$,  $(\wh
X^{ni},K^{ni})= \mathcal{SP}(A,\Pi;\wh Y^{ni})$. If
$\{Y^{ni}=H^{ni}+Z^{ni}\}$, $\{\wh Y^{ni}=\wh H^{ni}+\wh Z^{ni}\}$
with $H^{ni}_0,\wh H^{ni}_0\in\ov{ D(A)}$  and $Z^{ni}_0=\wh
Z^{ni}_0=0$, and $\{H^{ni}\}$, $\{\wh H^{ni}\}$ are tight in $\D$,
$\{Z^{ni}\}$, $\{\wh Z^{ni}\}$ satisfy (UT) and
\[
\lim_{i\to\infty}\limsup_{n\to\infty}
P(\sup_{t\leq T}|Y^{ni}_t-\wh Y^{ni}_t|\geq\epsilon)=0,
\quad T\in\Rp,\,\epsilon>0
\]
then
\[
\lim_{i\to\infty}\limsup_{n\to\infty}
P(\sup_{t\leq T}|X^{ni}_t-\wh X^{ni}_t|\geq\epsilon)=0,
\quad T\in\Rp,\,\epsilon>0.
\]
\end{corollary}
\begin{proof}
By Lemma \ref{lem4.1}, for every $T\in\Rp$ the arrays
$\{\updownarrow{\hspace{-0.1cm}K^{ni}\hspace{-0.1cm}}\updownarrow_{T}\}$
and $\{\updownarrow{\hspace{-0.1cm}\wh
K^{ni}\hspace{-0.1cm}}\updownarrow_{T}\}$ are bounded in
probability. Therefore the corollary follows immediately from
Remark \ref{rem2.6}(b).
\end{proof}

\begin{lemma}
\label{lem4.13} Let $\{Y^{n}\}$ be a  sequence of c\`{a}dl\`{a}g
processes such that $
Y_{0}^{n}\in\overline{\mathrm{D}(A)}$, $n\in{\mathbb{N}}$, and let $%
\{(X^{n},K^{n})\}$ be a sequence of solutions of the Skorokhod
problem associated with $\{Y^{n}\}$, i.e.
$(X^n,K^n)=\mathcal{SP}(A,\Pi;Y^n)$, $n\in{\mathbb{N}}$. Then for
any sequences $\{Z^{n}\}$ and $\{H^{n}\}$,
\begin{enumerate}
\item[\rm(i)] if $\displaystyle{
{\big\{}(Y^{n},H^{n},Z^{n}){\big\}}\quad\text{is tight in }{\mathbb{D}}({%
\mathbb{R}^{+}},{\mathbb{R}}^{2d})
}$
then%
\begin{equation*}
{\big\{}(X^{n},Y^{n},H^{n},Z^{n}){\big\}}\quad\text{is tight in }{\mathbb{D}}%
({\mathbb{R}^{+}},{\mathbb{R}}^{4d}).
\end{equation*}
\item[\rm(ii)]  if $\displaystyle{
(Y^{n},H^{n},Z^{n}){\xrightarrow[\mathcal{D}]{\;\;\;\; \;\;\;\;}}(Y,H,Z)\quad%
\text{in }{\mathbb{D}}({\mathbb{R}^{+}},{\mathbb{R}}^{3d})}$
then%
\begin{equation*}
(X^{n},Y^{n},H^{n},Z^{n}){\xrightarrow[\mathcal{D}]{\;\;\;\; \;\;\;\;}}%
(X,Y,H,Z)\quad\text{in }{\mathbb{D}}({\mathbb{R}^{+}},{\mathbb{R}}^{4d}),
\end{equation*}
where $(X,K)$ is a solution of the Skorokhod problem associated
with $Y$.
\end{enumerate}
\end{lemma}
\begin{proof}
It suffices to combine the deterministic results given in Remark
\ref{rem2.6}(d) with the Skorokhod representation theorem.
\end{proof}
\medskip

We say that MSDE (\ref{eq1.1}) has a weak solution if there exists a
 space $(\widehat{\Omega},\widehat{\mathcal{F}},\,(\widehat {%
(\mathcal{F}}_{t})_{t\geq0}),\widehat{\mathbb{P}})$ and $\widehat{\mathcal{F}}_{t}$%
-adapted processes $\widehat{H}$,$\widehat{Z}$ and $(\widehat{X},\widehat {K%
})$ such that $\mathcal{L}(\widehat{H},\widehat{Z})=\mathcal{L}(H,Z)$ and $(%
\widehat{X},\widehat{K})$ is a solution of the Skorokhod problem
associated with
\begin{equation*}
\widehat{Y}_{t}=\widehat{H}_{t}+\int_{0}^{t}f(\widehat{X}_{s-})\,d\widehat {Z%
}_{s},\quad t\in{\mathbb{R}^{+}}.
\end{equation*}
If ${\cal L}(\wh X,\wh K)={\cal L}(\wh X',\wh K')$ for any two
weak solutions $(\wh X,\wh K)$,  $(\wh X',\wh K')$ of the MSDE
(\ref{eq1.1}),   possibly defined on two different probability
spaces, then we say that the weak uniqueness for the MSDE
(\ref{eq1.1}) holds.

Let $\{ H^n \}$ be a sequence of $({\cal F}^n_t)$-adapted processes
such that $H^n_0\in \bar{D}$, $n \in \N$, and let $\{Z^n\}$ be a
sequence of $({\cal F}^n_t)$-adapted semimartingales satisfying (UT)
and such that $Z^n_0=0$, $n \in \N$. We consider the following the
sequence of $d$-dimensional MSDEs driven by the operator $A$ and
associated to
the non-expanding projection $\Pi$:%
\begin{equation}
X^n_{t}+K^n_{t}=H^n_{t}+\int_{0}^{t}\langle
f^n(X^n_{s-}),dZ^n_{s}\rangle,\quad t\in%
\mathbb{R}^{+},   \label{eq4.07}
\end{equation}
where
$f^n:\mathbb{R}^{d}\rightarrow\mathbb{R}^{d}\otimes\mathbb{R}^{d}$
is a continuous function. We  will need the following hypothesis.

\begin{description}
\item[(H5)]
   $f^n:{\mathbb{R}}^{d}\rightarrow
\mathbb{R}^{d}\otimes{{\mathbb{R}^{d}}}$ satisfies (H3)  for every
$n\in\N$ and there exists  $f:{\mathbb{R}}^{d}\rightarrow
\mathbb{R}^{d}\otimes{{\mathbb{R}^{d}}}$  such that $ \sup_{x\in
K}\|f^n(x)-f(x)\|\lra0 $ for every compact subset $K \subset \Rd
$.
\end{description}

We can now formulate our main stability result.

\begin{theorem}
\label{tw3} Assume (H1), (H2) and (H5). Let $\{ H^n \}$ be a
sequence  of $({\cal F}^n_t)$-adapted processes such that $H^n_0\in
\ov{D(A)}$, $n \in \N$, and let $\{Z^n\}$ be a sequence of  $({\cal
F}^n_t)$-adapted semimartingales satisfying (UT) and such that
$Z^n_0=0$, $n \in \N$. Let $\{(X^n,K^n)\}$ be a sequence of
solutions of the  MSDE (\ref{eq4.07}). If $(H^n,Z^n)\arrowd (H,Z)$
in $\Dii$ then
\begin{enumerate}
\item[{\rm(i)}] $\{(X^n,K^n, H^n, Z^n)\}$ is
tight in $ \Diiii \;$ and its every limit point is a weak solution
of the MSDE (\ref{eq1.1}),
\item[{\rm(ii)}]
if moreover (\ref{eq1.1}) has a unique weak solution $(X,K)$ then
\[
(X^n,K^n)\arrowd (X,K)\quad\mbox{\rm  in}\,\,\Dii.
\]
\end{enumerate}
\end{theorem}
\begin{proof} We follow the proof of   \cite[Theorem 4]{sl/93}.
First we show that for every  $T\in\Rp$,
\begin{equation}
\label{eq3.6a} \{\|X^n\|_T\}\quad\mbox{\rm is bounded in
probability}.
\end{equation}
Let $\wh X^n$ denote  the solution of  (\ref{eq4.07}) with
$f^n=0$, $n\in\N$. Since $\{ H^n \}$ is tight in $\D$, it follows
from Lemma \ref{lem4.1} that $ \{\|\wh{X}^n\|_T\}$ is bounded in
probability for every $T\in\Rp$. On the other hand, since
$\{Z^n\}$ satisfies  (UT), we may and will assume that
$Z^n_t=M^n_t+V^n_t$, $M^n_0=V^n_0=0$, where $\{[M^n]_T\}$,
$\{\updownarrow{\hspace{-0.1cm}
V^{n}\hspace{-0.1cm}}\updownarrow_{T}\}$ are bounded in
probability  and $|\Delta M^n|\leq 4c$, $|\Delta V^n|\leq c$ for some $c>0$. In this case
$\{\langle M^n\rangle_T\}$ is bounded in probability, as well.
Define $ \beta_k^n=\inf\{t; |\wh X^n_t|\vee
\updownarrow{\hspace{-0.1cm}V^{n}\hspace{-0.1cm}}\updownarrow_{t}
\vee\,[M^{n}]_t\vee\, \langle M^{n}\rangle _t>k\}\wedge k, $
$n,k\in\N$. It is clear that
\begin{equation}
\label{eq4.15}
\lim_{k\rightarrow+\infty}\limsup_{n\rightarrow+\infty}P(\beta^n_k\leq
T)=0,\quad T\in\Rp.
\end{equation}
Arguing as in Step 2 of the proof of Theorem \ref{thm4.5} we show
that $
\mathbb{E}\sup_{t<\beta^n_{k}}|X^n_{t}-\wh{X}^n_{t}|^{2}\leq
C(k,L)\exp \{3k\;C(k,L)\}$ for $n,k\in\N$, which together with
(\ref{eq4.15}) implies (\ref{eq3.6a}). Combining (\ref{eq3.6a})
with  (H5)  shows that  $\{ \sup_{t \leq T }\|f^n(X^n_{t-})\| \}$
is also bounded in probability for $T\in\Rp$. Therefore the
sequence of stochastic integrals $\{\int_0^{\cdot}\langle
f^n(X^n_{s-}) , dZ^n_s \rangle \}$ satisfies (UT).

Using  arguments from the proof  \cite[Theorem 4]{sl/93}  for
every $i\in\N$ one  can construct a sequence  $\{H^{ni}\}$ of
${\cal F}^n_t$-adapted processes such that for every $T\in\Rp$ the
sequence
$\{\updownarrow{\hspace{-0.1cm}H^{ni}\hspace{-0.1cm}}\updownarrow_{T}\}$
is bounded in probability and
\begin{equation} \label{ni} \lim_{i \rightarrow
\infty }\limsup_{n \rightarrow \infty }
 \p(\sup_{t \leq T}|H^{ni}_t-H^n_t| \geq \epsilon )=0, \quad \epsilon >0,\; T \in \Rp.
 \end{equation}
If  $ (X^{ni},K^{ni})=\mathcal{SP}(A,\Pi; H^{ni} +
\int_0^{\cdot}\langle f(X^n_{s-}) , dZ^n_s \rangle )$, $n,i\in\N$,
then for every $i\in\N$ the sequence $ \{X^{ni}\}$ satisfies  (UT)
(as a sum of three sequences satisfying (UT)) and by (\ref{ni})
and Corollary \ref{cor2.8},
 \[ \lim_{i
\rightarrow \infty }\limsup_{n \rightarrow \infty }
 P(\sup_{t \leq T}|X^{ni}_t-X^n_t| \geq \epsilon )=0,\quad \epsilon >0,\; T\in\Rp.\]  Furthermore,  it is well
known that for continuous $f: \Rd \lra \Rd \otimes \Rd $ one can
construct  a sequence of functions $\{ g^i \}$such that $g^i \in
\cdw $, $i\in\N$  and  $\sup_{x \in {K}}||g^i(x) - f(x) || \lra 0$
 for any compact subset $K \subset \Rd $. Set
\[
Y^{ni} =  H^{n}+ \int_0^{\cdot}\langle g^i(X^{ni}_{s-}) , dZ^n_s
\rangle , \qquad Y^{n} = H^{n}+ \int_0^{\cdot}\langle
f^n(X^{n}_{s-}) , dZ^n_s \rangle, \quad n,i \in \N.
\]
Since $\lim_{i\to\infty}\lim_{n\to\infty}\sup_{x \in {K}}\|g^i(x)
- f^n(x) \|= 0$ for any compact subset $K \subset \Rd $  and
$\{Z^n\}$ satisfies (UT), it is clear that
\begin{equation}\label{n4}
\lim_{i \rightarrow \infty }\limsup_{n \rightarrow \infty
} P(\sup_{t \leq T}|Y^{ni}_t-Y^n_t| \geq \epsilon )=0, \quad\epsilon
 > 0,\quad T\in\Rp.
\end{equation}
Fix $i\in\N$. From the fact that $g^i \in \cdw $  and $\{X^{ni}\}$ satisfies
(UT) we deduce that $\{ g^i(X^{ni}) \}$  satisfies (UT) as well. By \cite[Lemma 4]{sl/93} the sequence $\{ (Y^{ni},H^n,Z^n) \}$  is  tight in
$\Diii$, and hence, by (\ref{n4}),  $\{ (Y^n,H^n,Z^n) \}$ is tight
in $\DDD$. Since $(X^n,K^n)=\mathcal{SP}(A,\Pi; Y^n)$, it follows
from Lemma \ref{lem4.13} that
\begin{equation}
\label{RZZ} \{ (X^n,K^n,H^n,Z^n) \} \mbox{ is tight in  }\Diiii.
\end{equation}
The rest of the  proof of (i) runs as  the proof of \cite[Theorem 4]{sl/93}. Part (ii) follows immediately from (i).
\end{proof}

\begin{corollary}
\label{thm4.14}Under the assumptions (H1)--(H3) there exists a
weak solution $(X,K)$ of the MSDE (\ref{eq1.1}).
\end{corollary}

\begin{proof}
Let $\{f^{n}\}$ be a sequence of functions satisfying (H5) and
such that such that $f^{n}\in\mathcal{C}^{2}$, $n\in\N$.  By
Theorem \ref{thm4.5} for every $n\in{\mathbb{N}}$ there exists a
unique strong solution of the equation
\begin{equation*}
X_{t}^{n}+K_{t}^{n}=H_{t} +\int_{0}^{t}\langle
f^{n}(X_{s-}^{n}),\,dZ_{s}\rangle, \quad
t\in{\mathbb{R}^{+}},\,n\in\N,
\end{equation*}
so the desire result follows from Theorem  \ref{tw3}(i).
\end{proof}
\medskip

In the rest of this section we consider the convergence in
probability of solutions of  MSDEs. We will assume that the limit
MSDE (\ref{eq1.1}) has the pathwise uniqueness property, i.e.  for
any two  solutions $(\wh X,\wh K)$, $(\wh X',\wh K')$ of the  MSDE (\ref{eq1.1})
corresponding to processes $(\wh H,\wh Z)$, $(\wh H',\wh Z')$ and
defined on some probability space $\spawh$ with filtration
$\filtwh$, if $\wh P((\wh H_t,\wh Z_t)=(\wh H'_t,\wh
Z'_t);t\in\Rp)=1$ then $\wh P((\wh X_t,\wh K_t)=(\wh X'_t,\wh K'_t);t\in\Rp)=1$. Note
that using arguments  from the proof of Theorem \ref{thm4.5} one
can show that (H1)--(H4) imply pathwise uniqueness for (\ref{eq1.1}).

\begin{corollary} Assume that $(H^n,Z^n)\arrowp (H,Z)$ in $\DD$ and (\ref{eq1.1}) has
the pathwise uniqueness property. Then under  the assumptions of  Theorem \ref{tw3}
\begin{enumerate}
\item[{\rm(i)}]$\displaystyle{ (X^n,K^n,H^n,Z^n) \arrowp
(X,K,H,Z)}$ in $\Diiii$,
\item[{\rm(ii)}] if  $\sup_{t\leq T}|H^n_t-H_t|\arrowp0,\,\sup_{t\leq
T}|Z^n_t-Z_t|\arrowp0,\,T\in\Rp$   then
\[
\sup_{t\leq T}|X^n_t-X_t|\arrowp0\quad and\,\,\sup_{t\leq
T}|K^n_t-K_t|\arrowp0,\quad T\in\Rp,
\]
\end{enumerate}
where $(X,K)$ is a unique  strong solution of the MSDE (\ref{eq1.1}).\label{cor4.3}
\end{corollary}
\begin{proof} (i) It suffices to use   Theorem \ref{tw3}
and repeat arguments from the proof of  \cite[Corollary 11]{sl/93}.

(ii) By part (i), $(X^n,K^n,H^n,Z^n) \arrowp (X,K,H,Z) $ in
$\Diiii$. Moreover,
$\Delta X_t+\Delta K_t=\Delta H_t+f(X_{t-})\Delta Z_t$
and if $\Delta X_t\neq0$  or $\Delta K_t\neq0$  then  $\Delta
H_t\neq0$  or  $\Delta Z_t\neq0$. Therefore applying
\cite[Corollary C]{s1} gives (ii).
\end{proof}

\section*{Acknowledgments}

The work of L.M. and A.R.  was supported by the project ERC-like, code
1ERC/02.07.2012. The work of L.S. was supported by Polish NCN grant no.  2012/07/B/ST1/03508.
We are very grateful to the anonymous referee for his/her comments.

\end{document}